\documentclass{article}
\usepackage[margin=1in]{geometry}

\usepackage{graphicx} 
\usepackage{subcaption}
\usepackage{tikz}
\usepackage{float}
\usetikzlibrary{trees}
\usepackage{natbib}
\usepackage{amsfonts}
\usepackage{mathtools}

\usepackage{stmaryrd}
\usepackage{soul}
\usepackage{color}
\usepackage{times}
\usepackage{fancyhdr,graphicx,amsmath,amssymb}
\usepackage{amsthm}
\usepackage{bbm}
\usepackage{dsfont}
\usepackage{makecell}
\usepackage{ulem}
\usepackage[ruled,vlined]{algorithm2e}

\usepackage[utf8]{inputenc} 
\usepackage[T1]{fontenc}    
\usepackage{hyperref}       
\usepackage{url}            
\usepackage{booktabs}       
\usepackage{amsfonts}       
\usepackage{nicefrac}       
\usepackage{microtype}      
\usepackage{xcolor}         

\DeclareMathOperator{\argmin}{argmin}

\DeclareMathOperator{\bary}{F}

\DeclareMathOperator{\supp}{Supp}
\DeclareMathOperator{\Bary}{F}

\DeclareMathOperator{\card}{Card}

\usepackage{amsthm}
\usepackage{xcolor}

\newcommand{\R}{\mathbb R}

\newcommand{\Br}{\mathbb{B}_{R}}

\newtheorem{lem}{Lemma}[section]

\newtheorem{thm}[lem]{Theorem}
\newtheorem{rem}[lem]{Remark}

\newtheorem{prop}[lem]{Proposition}

\newtheorem{cor}[lem]{Corollary}

\usepackage{capt-of}
\usepackage{bbm}
\usepackage{hyperref}
\numberwithin{equation}{section}
\title{Sample complexity of optimal transport barycenters with discrete support}
\author{
  Léo Portales\\
  IRIT, TSE, INP Toulouse\\
  \url{leo.portales@irit.fr}
  \and
  Edouard Pauwels\\
  TSE, Université Toulouse Capitole\\
  \url{edouard2.pauwels@ut-capitole.fr}
  \and
  Elsa Cazelles\\
  CNRS, IRIT, Université de Toulouse\\
  \url{elsa.cazelles@irit.fr}
}

\date{October 2025}

\begin{document}

\maketitle

\begin{abstract}
Computational implementation of optimal transport barycenters for a set of target probability measures requires a form of approximation, a widespread solution being empirical approximation of measures.
We provide an $O(\sqrt{N/n})$ statistical generalization bound for the empirical sparse optimal transport barycenters problem, where $N$ is the maximum cardinality of the barycenter (sparse support) and $n$ is the sample size of the target measures empirical approximation. Our analysis includes various optimal transport divergences including Wasserstein, Sinkhorn and Sliced-Wasserstein. We discuss the optimality of our bound as well as the application of our results to specific settings such as K-means, constrained K-means, free and fixed support Wasserstein barycenters.
\end{abstract}

\section{Introduction}


Optimal transport barycenters provide a geometrically structured notion of mean for a set of probability measures. They are minimizers of functionals of the form
\begin{equation}\label{eq:intro}
    \nu\longmapsto\frac{1}{L}\sum_{\ell=1}^L\mathbb{D}\left(\mu^\ell,\nu\right),
\end{equation} 
where $\nu$ is a probability measure, $\mathbb{D}$ is an optimal transport divergence, and $\left\{\mu^\ell\right\}_{\ell=1}^L$ is a collection of target probability measures in $\R^d$. 
This problem was first introduced for $\mathbb{D}=W_2^2$ \cite{agueh2011barycenters}, as a generalization of McCann's interpolation \cite{mccann1997convexity}. This was also proposed independently for texture synthesis and mixing \cite{rabin2011wasserstein}. Following fast implementation of $W_2$-barycenters (e.g. \cite{cuturi2014fast}), they became a central tool in machine learning and statistics with various applications such as sample blending \cite{srivastava2015wasp}, texture mixing \cite{rabin2011wasserstein}, multisource domain adaptation \cite{montesuma2021wasserstein} or music genre recognition \cite{montesuma2021wasserstein2}. Various extensions based on different transport divergences $\mathbb{D}$ in \eqref{eq:intro} exist: entropic \cite{chizat2023doubly}, Bures \cite{chewi2020gradient}, sliced \cite{bonneel2015sliced} and Gromov Wasserstein barycenters \cite{brogat2022learning}. 

\paragraph{Sample averaged approximation and sparse barycenters}
 In many applications, the target probability measures in \eqref{eq:intro} are computationally intractable, and one needs to resort to discrete approximations. Typical schemes include histogram approximation (e.g. images), sample average approximation (SAA) \cite{montesuma2021wasserstein,claici2018stochastic} or online stochastic approximation (SA) \cite{dvurechenskii2018decentralize,dvinskikh2020stochastic,backhoff2024stochastic}.  We focus on sample average approximation arising in two main use cases: the statistical setting where distributions are only accessible through empirical samples \cite{montesuma2021wasserstein} and numerical Monte-Carlo approximation of known densities \cite{puccetti2020computation}. 
In both cases, all measures, including the resulting barycenter, are discrete, possibly supported on a very large number of atoms.
Consequently, most optimal transport barycenter algorithms  restrict the number of atoms of the barycenter, similarly to optimal quantization of measures \cite{pages2004optimal}. Our analysis includes this support constraint.  In other words, the probability measures $\nu$ in \eqref{eq:intro}, is subject to the constraint $\vert\supp(\nu)\vert\leq N$ for some integer $N\geq 1$, where $\supp(\nu)$ denotes the support of $\nu$. This problem is known as  optimal transport barycenters with sparse support \cite{yang2024approximate,borgwardt2021computational}. Many modern machine learning problems leverage this sparse barycenter framework, including applications in LLM \cite{ai2025resmoe}, flow cytometry \cite{erell2024}, graph clustering \cite{etienne2025pasco} or cross-machine fault diagnosis \cite{lee2025barycenter}. 

Our main problem of interest is that of statistical guaranties for the optimal transport barycenters sample averaged approximation under support cardinality constraints, with various divergences $\mathbb{D}$: (sliced) Wasserstein distances, (debiased) Sinkhorn divergences. The collection of target probability measures $\{\mu^\ell\}_{\ell=1}^L$ in $\R^d$ each relate to $n$ \textit{i.i.d} random variables $X_1^\ell,\ldots,X_n^\ell\overset{\text{i.i.d}}{\sim}\mu^\ell$ with empirical measures by $\mu_n^\ell:=\frac{1}{n}\sum_{i=1}^n\delta_{X_i^\ell}$, and the empirical sparse barycenters are given by
\begin{equation}\label{eq:Sparse_bary}
\underset{(Y,\pi)\in(\R^d)^N\times \Delta_N}{\argmin} \ \frac{1}{L}\sum_{\ell=1}^L\mathbb{D}\left(\mu_n^\ell,\sum_{i=1}^N\pi_i\delta_{y_i}\right)
\end{equation}
where $Y:=(y_1,\ldots, y_N)$ is a point cloud in $\R^d$ and $\Delta_N$ is the $(N-1)$-probability simplex.

\paragraph{Statistical guaranties for optimal transport barycenters}
For general population measures, statistical rates of the form $O(n^{-1/3d})$ for empirical barycenters  were described in \cite{carlier2024quantitative}, suffering from the curse of dimensionality. On the other hand, for discrete target measures, an $O(n^{-1/2})$ rate is known \cite{heinemann2022randomized}. This is reminiscent of the statistical complexity of optimal transport adapting to the least complex measure \cite{hundrieser2024empirical}. Indeed, empirical optimal transport converges as $O(n^{-1/d})$ \cite{fournier2015rate}, while if at least one of the two measures is discrete, the rate is $O(n^{-1/2})$ \cite{del2024central}.

\paragraph{Contribution}
We provide uniform rates in expectation for the empirical sparse optimal transport barycenter functional \eqref{eq:Sparse_bary}. They are of the form  $O(\sqrt{N/n})$, where the barycenter has $N$ support points, and $n$ is the number of samples per empirical target measure. We consider several divergences bridging the gap between various results in the literature. We deduce sample complexity estimates for the optimal transport barycenters with sparse support. Our proof techniques combine tools from empirical risk minimization theory and estimations on dual variables arising in the dual formulations of transport divergences.

\paragraph{Notations}
Throughout this article we will denote respectively $\mathcal{M}_1(\Omega)$ and $\mathcal{M}_1^N(\Omega)$, for some $\Omega\subset\R^d$, the set of probability measures on $\Omega$  and the set of probability measures supported on at most $N$ points in $\Omega$. We will denote $\Br$ the closed ball of $\R^d$ centered in $0$ and of radius $R>0$. Additionally, $\|\cdot\|$ is the Euclidean norm on $\R^d$, $\Delta_N$ the $N-1$ probability simplex and $\mathbb{S}^{d-1}$ the $(d-1)$-sphere in $\R^d$. For a measure $\mu$ and a measurable map $T$ we denote by $T_{\sharp}\mu$ the pushforward (or image measure) of $\mu$ by $T$.
\paragraph{Organization of the paper} 
We introduce technical material on optimal transport divergences in Section \ref{sec:ot_and_barycenters}. Computational aspects are discussed in Section \ref{sec:OT_algo}. Our main results are presented in Section \ref{sec:main_results}, with a discussion in Section \ref{sec:discussion}. Dual formulations of the divergences considered, together with technical details and proofs, are postponed to the appendices.

\section{Introduction to optimal transport and barycenters}\label{sec:ot_and_barycenters}

\subsection{Optimal transport divergences}

We refer to \cite{peyre2019computational} for a computational point of view, and to \cite{villani_opt_old_new} for a more theoretical presentation. The optimal transport between two measures $\mu,\nu\in\mathcal{M}_1(\R^d)$ is a solution to the optimization problem over measures: 
\begin{equation}\label{eq:OT}
    \inf_{\gamma\in \Pi(\mu,\nu)}\ \int_{\R^d \times \R^d} c(x,y) d\gamma(x,y),
\end{equation}
where $c:\R^d\times\R^d\to\R$ is a lower semi-continuous cost function and $\Pi(\mu,\nu)$ is the set of couplings: probability measures with marginals $\mu$ and $\nu$. In the Euclidean case, if $c\colon (x,y) \mapsto \|x-y\|^p$ for an integer $p\geq 1$, this is the Wasserstein distance $W_p$ on $\R^d$:
\begin{equation}\label{eq:OT_wass}
W_p^p(\mu,\nu) = \inf_{\gamma\in \Pi(\mu,\nu)}\ \int_{\R^d \times \R^d} \Vert x-y\Vert^p d\gamma(x,y).
\end{equation}
The semi-discrete setting corresponds to one measure being finitely supported. We focus on this setting which encompasses the sparse barycenter problem.  The infinite dimensional problem over measures can be reformulated as a finite-dimensional maximization problem based on optimal transport duality (see Appendix \ref{Semi-dual formulation}).
Various optimal transport-based divergences were introduced based on variations of the Wasserstein distance.

The Sinkhorn divergences \cite{cuturi2013sinkhorn,peyre2019computational} add an entropic regularization term, for  $\mu,\nu\in\mathcal{M}_1(\R^d)$:\begin{equation}\label{eq:OT_eps}
W_{\epsilon,p}^p(\mu,\nu) = \inf_{\gamma\in \Pi(\mu,\nu)}\int_{\R^d\times \R^d} \Vert x-y\Vert^p d\gamma(x,y)+\epsilon\text{KL}(\gamma|\mu\otimes\nu)
\end{equation}
where $\text{KL}(\gamma \vert \xi) := \int_{\R^{2d}\times \R^{2d}} (\log(\frac{d\gamma}{d\xi}(x,y))-1)d\gamma(x,y)$ denotes the Kullback-Leibler divergence and $\epsilon>0$ is a regularization parameter. For discrete measures, the optimization problem in \eqref{eq:OT_eps} can then be efficiently solved using the Sinkhorn algorithm \cite{sinkhorn1967concerning, cuturi2013sinkhorn}. Similarly as Wasserstein distances, in the semi-discrete setting, there is a finite dimensional dual problem \cite{genevay2016stochastic} of particular use in our analysis, it is presented in Appendix \ref{Semi-dual formulation}. 

Considering projections onto lines leads to sliced Wasserstein distances \cite{rabin2011wasserstein,kolouri2019generalized}. For any $\theta\in\mathbb{S}^{d-1}$, set $P_{\theta}:x\in\mathbb{R}^{d}\mapsto  \langle \theta,x \rangle\in \R$, understood as the projection on the line directed by $\theta$. For $\mu,\nu \in\mathcal{M}_1(\R^d)$, one obtains two collections of pushforward univariate measures $(P_{\theta \sharp}\mu)_{\theta\in\mathbb{S}^{d-1}}$ and $(P_{\theta \sharp}\nu)_{\theta\in\mathbb{S}^{d-1}}$.  Sliced distances combine divergences between these univariate measures for all directions $\theta\in \mathbb{S}^{d-1}$. We obtain the Sliced and max-Sliced Wasserstein distances respectively defined as 
\begin{equation}
\begin{split}
SW_{p}^{p}(\mu,\nu)&=\int_{\mathbb{S}^{d-1}}W_{p}^{p}({P_{\theta}}_{\sharp} \mu,{P_{\theta}}_{\sharp} \nu)d\sigma(\theta)\\
\text{max-}SW_{p}^{p}(\mu,\nu)&=\max_{\theta\in\mathbb{S}^{d-1}}W_{p}^{p}({P_{\theta}}_{\sharp} \mu,{P_{\theta}}_{\sharp} \nu),
\end{split}
\label{eq:sliced_max}
\end{equation}
where $\sigma$ denotes the uniform probability measure over $\mathbb{S}^{d-1}$. 
In the semi-discrete setting, these also admits finite dimensional dual problems as discussed in Appendix \ref{Semi-dual formulation}.

\subsection{Optimal transport barycenters with sparse support}
Given a collection of target probability measures $\{\mu^\ell\}_{\ell=1}^L$, an unconstrained barycenter with respect to divergence $\mathbb{D}$ is a minimizer of \eqref{eq:intro}, it is typically absolutely continuous \cite{agueh2011barycenters,brizzi2025p} if at least one of the target measures is absolutely continuous.
A common strategy is approximation by discrete probability measures, restricting the optimization set in \eqref{eq:intro} to $\mathcal{M}_1^N(\R^d)$. 
This is referred to as optimal transport barycenter problem with \emph{sparse support}, with underlying measures of the form $\nu^*:=\sum_{i=1}^N\pi_i\delta_{y_i}$ where $Y:=(y_1,\ldots,y_N)\in(\R^d)^N$ and $\pi\in\Delta_N$. The problem is reformulated as follows:
\begin{equation}\label{general_problem}
    \underset{(Y,\pi)\in A}{\min}\ \bary_\mathbb{D}\left(\mu^1,\ldots,\mu^L,\sum_{i=1}^N\pi_i\delta_{y_i}\right):= \underset{(Y,\pi)\in A}{\min}\ \frac{1}{L}\sum_{\ell=1}^L \mathbb{D}\left(\mu^\ell,\sum_{i=1}^N\pi_i\delta_{y_i}\right)
\end{equation}
where A is a closed nonempty subset of $(\R^d)^N\times \Delta_N$ which integrates possibly additional constraints. Formulation \eqref{general_problem} covers various settings described in the literature \cite{cuturi2014fast,borgwardt2021computational}:

\vspace{0.2cm}\begin{center}
\begin{tabular}{ |>{\centering\arraybackslash}p{2.5cm}|>{\centering\arraybackslash}p{5.5cm}|>{\centering\arraybackslash}p{6cm}| } 
\hline
$A$ & $L=1$ & $L>1$ \\
\hline
$(\R^d)^N\times \Delta_N$ &  Optimal quantization  & Barycenter with sparse support \\ 
$ (\R^d)^N\times \{\bar{\pi}\}$&  Constrained quantization & Free support barycenter \\ 
$\{\bar{Y}\}\times\Delta_N$& \raisebox{0.5ex}{\rule{0.5cm}{0.4pt}} & Fixed support barycenter\\ 
\hline
\end{tabular}
\end{center}\vspace{0.1cm}


\subsection{Discretization and computational complexity}\label{sec:OT_algo}

Even for continuous target measures, barycenter algorithms operate on discrete data. A common approach is to approximate the target measures by empirical averages, and to compute their barycenter \eqref{eq:intro} with dedicated methods \cite{cuturi2014fast,li2020continuous,heinemann2022randomized}. 
Most algorithms actually add support constraints as in \eqref{general_problem}, effectively computing a sparse barycenters  \cite{yang2024approximate,borgwardt2021computational,mimouni2025computation}. 
For discrete measures supported on $N$, the optimal transport problem has an $O(N^3\log(N))$ computational cost \cite{peyre2019computational}. The discrete Wasserstein barycenter problem inherits this complexity \cite{altschuler2022wasserstein,borgwardt2021computational}, making it tractable only for measures supported on a small number of atoms \cite{altschuler2021wasserstein}.
The complexity of Sinkhorn divergences $W_{\epsilon,p}^p$ \eqref{eq:OT_eps} is $O(N^2/\epsilon)$\cite{dvurechensky2018computational}, which allows for a trade-off between optimal transport accuracy and algorithmic efficiency \cite{chizat2023doubly}. 
Sliced-Wasserstein distances are based on univariate closed form quantile expressions \cite[Chapter 2.2]{ref_villani_topics}, making them also computationally relevant \cite{bonneel2015sliced}. 

\section{Sample complexity for sparse optimal transport barycenters }\label{sec:main_results}

This section presents our main results : statistical generalization bounds for optimal transport barycenters. We draw inspiration from \cite[Section 4.4]{gyorfi2002principles}, which treats the K-means functional (see Section \ref{sec:kmeans}). The following provides uniform generalization bounds for the empirical barycenter cost function, which most barycenter algorithms seek to minimize in practice (see Section \ref{sec:OT_algo}).

\begin{thm}[Generalization error bound]\label{th:main}
Let $\mu^1,\ldots,\mu^L\in \mathcal{M}_1(\Br)$, for $R>0$. For some integer $n$, let $\mu_{n}^1,\ldots,\mu_n^L$ be empirical measures supported over $n$ i.i.d random variables of respective law $\mu^\ell$, for all $\ell\in [\![1,L]\!]$. Let $\mathbb{D}=W_p^p,\:W_{\epsilon,p}^p$ (for some $\epsilon>0$), $\:SW_p^p$ or $\text{max-}SW_p^p$. Then for any integer $N\geq 1$, and $\bary_{\mathbb{D}}$ the function defined in \eqref{general_problem}, we have
\begin{equation*}
\mathbb{E}\left[\sup_{(Y,\pi)\in \Br^N\times\Delta_N}\left|\bary_\mathbb{D}\left(\mu^1,\ldots, \mu^L,\sum_{i=1}^N\pi_i\delta_{y_i}\right)-\bary_\mathbb{D}\left(\mu_n^1,\ldots, \mu_n^L,\sum_{i=1}^N\pi_i\delta_{y_i}\right)\right| \right]\leq C_{p,R}\sqrt{\frac{C_{d,N}}{n}},
\end{equation*}
where the finite constant values $C_{p,R}$ and $C_{d,N}$ depend on the divergence $\mathbb{D}$ and are summarized in Table \ref{tab:main_th}.
\begin{table}[ht]
    \renewcommand{\arraystretch}{2}
\setlength{\abovecaptionskip}{5pt}
    \centering
    \begin{tabular}{c|c|c}
        $\mathbb{D}$ &  $C_{p,R}$&  $C_{d,N}$ \\
        \hline
        {$W_p^p$}  & {$8\sqrt{2}\int_{0}^{2(2R)^p}\sqrt{\log\left(2+\frac{32pR^p}{\tau}\right)}d\tau$} & {$N(d+1)$}\\
        \makecell{{$SW_p^p$}\\ {$\text{max-}SW_p^p$}}&{$8\sqrt{2}\int_{0}^{2(2R)^p}\sqrt{\log\left(2+\frac{48pR^p}{\tau}\right)}d\tau$} & {$N(d+1)+d$}\\
        {$W_{\epsilon,p}^p$} & {$8\sqrt{2}\int_{0}^{(4p+1)(2R)^p}\sqrt{\log\left(2+\frac{64p(2R)^p}{\tau} \right)}d\tau$} & {$N(d+1)$}\\
    \end{tabular}    \vspace{0.2cm}\caption{Explicit constants arising in the generalization error bound for the divergences $\mathbb{D}.$}
    \label{tab:main_th}
\end{table}
\end{thm}
\begin{rem}[On the hypothesis of Theorem \ref{th:main}]
We consider empirical samples of equal size $n$ for each measure, as well as uniform weights $\frac{1}{L}$ for simplicity. Our results hold for different sample sizes, replacing $n$ by the minimal sample size, and any weights $\lambda\in \Delta_L$.
\end{rem}
\begin{proof}[Sketch of proof]
Complete details are provided in Appendix \ref{Main_th_proof_appendix}.
We treat here the case $\mathbb{D}=W_p^p$ for simplicity. Let $(Y,\pi)\in \Br^N\times\Delta_N$, the goal is to control, for each $\ell=1,\ldots,L$, the variation between empirical and population $W_p^p$ divergences.  
This is achieved by resorting to the dual formulation of $W_p^p$, in which the dual variables are vectors $w\in\mathbb{R}^N$. Lemma \ref{lem:ineq_sup} ensures that there exist functions $f_Y^w$ such that for each $\ell$ and any $\pi\in\Delta_N$
\begin{equation*}
\left| W_p^p\!\left(\mu_n^\ell,\sum_{i=1}^N\pi_i\delta_{y_i}\right)-W_p^p\!\left(\mu^\ell,\sum_{i=1}^N\pi_i\delta_{y_i}\right) \right| \leq \sup_{w\in\mathbb{R}^N}\left|\mathbb{E}_{X\sim\mu_n^\ell}\!\left[f_Y^w(X)\right]-\mathbb{E}_{X\sim\mu^\ell}\!\left[f_Y^w(X)\right]\right|.
\end{equation*}
Proposition \ref{Prop:barycentersproperties} provides a bound on the dual variables, $w\in[-K,K]^N$ for some constant $K$ independent of $Y$ and $\pi$. This allows us to invoke standard results from empirical process theory (e.g., \cite[Theorem 3.5.1]{gine2021mathematical}).  
\end{proof}

\begin{rem}
In the specific case where $\mathbb{D}$ is a distance, using the triangular inequality  we have $|\mathbb{D}(\mu_n^\ell,\nu)-\mathbb{D}(\mu^\ell,\nu)|\leq \mathbb{D}(\mu_n^\ell,\mu^\ell)$.
Statistical results for $\mathbb{D}$ \cite{fournier2015rate} could then be used to derive error estimates. However, for $\mathbb{D}=W_1$ for example, the estimate is of order $n^{-1/d}$ \cite[Proposition 2.1]{chewi2024statistical}, compared with $n^{-1/2}$ in Theorem \ref{th:main}. 
\end{rem}

\begin{cor}[Sparse optimal transport barycenters]\label{cor:sparse_ot_bar}
Let $\mu^1,\ldots,\mu^L\in \mathcal{M}_1(\Br)$, for $R>0$. For some integer $n$, let $\mu_{n}^1,\ldots,\mu_n^L$ be empirical measures supported over $n$ i.i.d random variables of respective law $\mu^\ell$, for all $\ell\in [\![1,L]\!]$. Let also $N\geq1$ be some integer,  $A$ a closed nonempty subset of $\Br^N\times \Delta_N$, and let $\mathbb{D}=W_p^p,\:W_{\epsilon,p}^p$ (for some $\epsilon>0$), $SW_p^p$ or $\text{max-}SW_p^p$. We denote by $(Y_n,\pi_n)$ an argmin in $A$ of the empirical barycenter cost $(Y,\pi)\mapsto \Bary_{\mathbb{D}}\left(\mu_n^1,\ldots,\mu_n^L,\sum_{i=1}^N\pi_i\delta_{y_i}\right)$ where $\Bary_{\mathbb{D}}$ is given in \eqref{general_problem}.
Then we have :
\begin{equation}\label{eq:cor:approx_error}
\mathbb{E}\left[\Bary_{\mathbb{D}}\left(\mu^1,\ldots,\mu^L,\sum_{i=1}^N\pi_i^n\delta_{y_i^n}\right)-\underset{(Y,\pi)\in A}{\min}\Bary_{\mathbb{D}}\left(\mu^1,\ldots,\mu^L,\sum_{i=1}^N\pi_i\delta_{y_i}\right) \right] \leq 2C_{p,R} \sqrt{\frac{C_{d,N}}{n}},
\end{equation}
where the constant $C_{p,R}$ and $C_{d,N}$ are summarized in the Table \ref{tab:main_th}.
\end{cor}

\begin{rem} 
    For $\mathbb{D}=W_2^2$ and $L=1$,  the estimation error in Corollary \ref{cor:sparse_ot_bar} is upper bounded by $O(N/\sqrt{n})$ \cite{biau2008performance}, independently of the dimension $d$, to be compared to our  $O(\sqrt{dN/n})$ results. 
\end{rem}
\begin{proof}
This leverages classical statistical learning theory arguments.
By continuity of the functionals
$(Y,\pi)\mapsto\bary_\mathbb{D}(\mu^1,\ldots,\mu^L,\sum_{i=1}^N\pi_i\delta_{y_i})$ and $(Y,\pi)\mapsto\bary_\mathbb{D}(\mu_n^1,\ldots,\mu_n^L,\sum_{i=1}^N\pi_i\delta_{y_i})$, both admit respective global minimizers in $A$,  $(Y^*,\pi^*)$ and $(Y_n^*,\pi_n^*)$, with associated discrete measures $\nu^*:=\sum_{i=1}^N\pi_i\delta_{y_i}$ and $\nu_n^*:=\sum_{i=1}^N\pi_i^n\delta_{y_i^n}$. Then we have
\begin{equation*}
\begin{split}
&\bary_\mathbb{D}(\mu^1,\ldots,\mu^L,\nu_n^*)-\bary_\mathbb{D}(\mu^1,\ldots,\mu^L,\nu^*)\\
&\leq \bary_\mathbb{D}(\mu^1,\ldots,\mu^L,\nu_n^*)-\bary_\mathbb{D}(\mu_n^1,\ldots,\mu_n^L,\nu_n^*)+\bary_\mathbb{D}(\mu_n^1,\ldots,\mu_n^L,\nu^*) -\bary_\mathbb{D}(\mu^1,\ldots,\mu^L,\nu^*)\\
&\leq 2\sup_{(Z,\tau)\in A}\left|\bary_\mathbb{D}\left(\mu_n^1,\ldots,\mu_n^L,\sum_{i=1}^N\tau_i\delta_{z_i}\right)-\bary_\mathbb{D}\left(\mu^1,\ldots,\mu^L,\sum_{i=1}^N\tau_i\delta_{z_i}\right) \right|.
\end{split}
\end{equation*}
Taking the expectation on both sides and applying Theorem \ref{th:main} allows to conclude.
\end{proof}

Theorem \ref{th:main} and the sample complexity results of Sinkhorn divergences in \cite{genevay2019sample}  also yield an upper bound for the debiased Sinkhorn barycenters \cite{genevay2018learning,janati2020debiased}, although with constants which have an exponential dependence in the dimension. The proof of the next corollary is also postponed in Appendix \ref{Main_th_proof_appendix}.

\begin{cor}[Debiased Sinkhorn barycenters]\label{cor:sinkhorn_debiased}
Let $\mu^1,\ldots,\mu^L\in \mathcal{M}_1(\Br)$, for $R>0$. For some integer $n$, let $\mu_{n}^1,\ldots,\mu_n^L$ be empirical measures supported over $n$ i.i.d random variables of respective law $\mu^\ell$, for all $\ell\in [\![1,L]\!]$. Let also $N\geq 1$ be an integer, $A$  a closed nonempty subset of $\Br^N \times \Delta_N $ and for any $\epsilon>0$, set $\overline{W}_{\epsilon,p}^p(\mu,\nu)=W_{\epsilon,p}^p(\mu,\nu)-\frac{1}{2}(W_{\epsilon,p}^p(\mu,\mu)+W_{\epsilon,p}^p(\nu,\nu))$. We denote by $(Y_n,\pi_n)$ an argmin in $A$ of the empirical barycenter cost $(Y,\pi)\mapsto \bary_{\overline{W}_{\epsilon,p}^p}\left(\mu_n^1,\ldots,\mu_n^L,\sum_{i=1}^N\pi_i\delta_{y_i}\right)$ where $\Bary_{\mathbb{D}}$ is given in \eqref{general_problem}.
Then we have
\begin{equation}\label{eq:cor:debiased}
\begin{split}
    \mathbb{E} \left[\bary_{\overline{W}_{\epsilon,p}^p}\left(\mu^1,\ldots,\mu^L,\sum_{i=1}^N\pi_i^n\delta_{y_i^n}\right) - \underset{(Y,\pi)\in A} {\min}\bary_{\overline{W}_{\epsilon,p}^p}\left(\mu^1,\ldots,\mu^L, \sum_{i=1}^N\pi_i\delta_{y_i}\right)\right]
    &\leq \frac{C_{\overline{W}}}{\sqrt{n}}
\end{split}
\end{equation}
with $C_{\overline{W}} = C_{p,R}N(d+1)+ 2Ke^\frac{{\mathcal{K}}}{\epsilon}\left(1+\epsilon^{-\lfloor d/2\rfloor}\right)$
with $\mathcal{K}=2\sqrt{2}p(2R)^{p-1}R+(2R)^p$, $K$ is a constant depending only on $R$ and $p$, and $C_{p,R}$ is the constant related to $W_{\epsilon,p}^p$ in Theorem \ref{th:main}.
\end{cor}

\section{Discussion of the results}\label{sec:discussion}

\subsection{Sample complexity of optimal transport divergences}


Empirical optimal transport suffers from the curse of dimensionality \cite{fournier2015rate}, but it is now well established that its statistical convergence depends on the intrinsic complexity of the target measures \cite{hundrieser2024empirical}. If one of them is discrete, the convergence rate is of actually of order $O(1/\sqrt{n})$ \cite{sommerfeld2018inference,genevay2019sample,hundrieser2024empirical,nadjahi2020statistical,lin2021projection}. Theorem \ref{th:main} allows to recover convergence rates for the quantity $\mathbb{E}\vert \mathbb{D}(\mu_n,\nu)-\mathbb{D}(\mu,\nu)\vert$, for any $\nu\in \mathcal{M}_1^N(\Br)$, as well as for the empirical divergences themselves. 
These do not improve upon the literature.

\subsection{Barycenters with sparse support}\label{sec:bary_discussion}

For a sparse barycenter with at most $N$ support points and sample size $n$, Theorem \ref{th:main} provides generalization bounds of order $\sqrt{N/n}$, classical in semi-discrete optimal transport \cite[Theorem 2.10]{del2024central}. Interestingly, this rate is uniform in the number of measures $L$, in the regularization parameter $\epsilon$ for the Sinkhorn divergence, without exponential dependency on the ambiant dimension $d$.

For $\mathbb{D}=W_2^2$, assuming that each measure $\mu^\ell$ is supported on $n^\ell$ points, and setting $N\geq \sum_{\ell=1}^Ln^\ell-L+1$, problems \eqref{eq:intro} and \eqref{eq:Sparse_bary} coincide: an unconstrained barycenter turns out to be sparse \cite[Theorem 2]{anderes2016discrete}. In this setting, \cite{heinemann2022randomized} obtains the same rate of convergence $\sqrt{N/\min_\ell (n^\ell)}$ for the unconstrained barycenter, this analysis being limited to large $N$.


\subsection{Tightness of the result}

In this section, we discuss the optimality of the estimation error bound in Corollary \ref{cor:sparse_ot_bar} with respect to the sample size $n$ and the sparsity level $N$.
\subsubsection{Lower bound of the estimation error}\label{sec:opt_in_n}
We aim to establish a lower bound for the estimation error, i.e., the left-hand side term in Corollary \ref{cor:sparse_ot_bar}, in a worst-case sense over the choice of target measures. 
We rely on \cite[Theorem 1]{bartlett2002minimax} which provides minimax lower bounds for the optimal quantization cost ($L=1$ and  $\mathbb{D}=W_2^2$). A direct extension of this result provides a lower bound for the sparse $2$-Wasserstein barycenter problem.

\begin{lem}\label{Lem:lowerbound}
Let $\mathbb{D}=W_2^2$, the number of measures $L\in \mathbb{N}^*$ and the dimension $d\in \mathbb{N}^*$. Assume  $N\geq 3$ and $n\geq 3435N$ and let  $\sum_{i=1}^N\pi_i^n\delta_{y_i^n}\in\mathcal M_1^N(\R^d)$ be a random measure depending on $nL$ random variables in $\Br$. Then
there exists $\mu^1,\ldots,\mu^L\in\mathcal{M}_1(\R^d)$ such that the estimation error of Corollary \ref{cor:sparse_ot_bar} verifies
\begin{equation}\label{eq:approximation_error}
\mathbb{E}\left[\bary_{\mathbb{D}}\left( \mu^1,\ldots,\mu^L,\sum_{i=1}^N\pi_i^n\delta_{y_i^n} \right)-\min_{(Y,\pi)\in \Br^N\times\Delta_N}\bary_{\mathbb{D}}\left( \mu^1,\ldots,\mu^L,\sum_{i=1}^N\pi_i\delta_{y_i} \right)\right]\geq C_{d}\sqrt{\frac{N^{1-\frac{4}{d}}}{n}}
\end{equation}
 where $C_{d}$ is an explicit constant which only depends on $d$.
\end{lem}
The proof of Lemma \ref{Lem:lowerbound} can be found in  \ref{appendix_sec_lem_lowerbound}.
For a fixed $N$, Lemma \ref{Lem:lowerbound} implies that the $O(1/\sqrt{n})$ convergence rate for barycenter functional in \ref{cor:sparse_ot_bar} cannot be improved without further assumptions. 
However the dependency in $N$ differs.

\subsubsection{Dependency in $N$ in the estimation error}

There is an $O\left(N^{-2/d}\right)$ factor between the lower bound in \eqref{eq:approximation_error} and the upper bound   in \eqref{eq:cor:approx_error}. This raises the question of the optimal exponent in $N$. We exhibit a situation where the lower bound is tight.

Setting $N=n$,  $L=1$  and $\pi=(1/n,\ldots,1/n)$, the upper bound \eqref{eq:cor:approx_error} of Corollary \ref{cor:sparse_ot_bar} remains constant as $n$ increases. However, we will argue that the error actually decreases like $O(n^{-2/d})$. Indeed, in this setting, the minimizer of $\mu\mapsto W_2^2(\mu_n,\mu)$ over  $\mathcal{M}_1^N(\R^d)$, is clearly $\mu_n$ and thus the estimation error writes
\begin{equation*}
\mathbb{E}\left[W_2^2(\mu,\mu_n)-\min_{\nu\in\mathcal{M}_1^n(\Br)}W_2^2(\mu,\nu)\right] \leq \mathbb{E}\left[W_2^2(\mu,\mu_n)\right] = O(n^{-2/d}),
\end{equation*}
see \cite{fournier2015rate}. 
Hence Corollary \ref{cor:sparse_ot_bar} is not optimal in this particular case. 
The above argument suggests that the optimal exponent in $N$ is related to estimates of the quantity $W_2(\mu,\mu_n)$, a notoriously difficult problem in statistical optimal transport \cite[Section 2]{chewi2024statistical}. The curse of dimensionality is in fact hidden in the sparsity level $N$, as remarked in \cite[Theorem 2.10]{del2024central}.

\subsubsection{Beyond the $W_2^2$ case}

We expect the same conclusions for $\mathbb{D} = SW_p^p$ and $\text{max-}SW_p^p$ since the $1/\sqrt{n}$ rate is also known to be optimal for these divergences \cite{xu2022central}. The bias of the Sinkhorn divergence $\mathbb{D}=W_{\epsilon,p}^p$ benefits from a more favorable $1/n$ convergence rate \cite{del2023improved} suggesting that a sharper rate is attainable for barycenters as well.

\subsection{K-means and constrained K-means}\label{sec:kmeans}
Choosing $\mathbb{D}=W_2^2$ and $L=1$, Corollary \ref{cor:sparse_ot_bar} provides error bounds for the celebrated K-means problems.
Let $X_1,\ldots,X_n$ be $n$ i.i.d samples with distribution $\mu\in \mathcal{M}_1(\Br)$. It is well known minimizing the K-means functional is actually an optimal quantization problem for the empirical measure $\mu_n$ (see Appendix \ref{K-means=Opt_quant}):
\begin{equation}
    \label{k-means_metric}
    \min_{Y\in (\Br)^{N}}\frac{1}{n}\sum_{i=1}^n\min_{j=1,\ldots,N}\|X_i-y_j\|^2=\min_{\nu\in\mathcal{M}_1^N(\Br)}W_2^2(\mu_n,\nu).
\end{equation}
Corollary \ref{cor:sparse_ot_bar} then yields a convergence rate of order $O(\sqrt{N/n})$ for the K-means problem, which yields an improvement by a factor $\sqrt{\log(N)}$ compared to \cite[Section 4.4]{gyorfi2002principles}. The Wasserstein functional formulation  is closer to the measure approximation view point. It also allows to consider the constrained K-means problem \cite{ng2000note}, fixing the mass associated to each centroid with a weight vector $\bar{\pi}\in\Delta_N$ and minimizing the functional $Y \mapsto W_2^2\left(\mu_n,\sum_{i=1}^N\bar{\pi}_i\delta_{y_i}\right)$.
For uniform weights $\bar{\pi}_i=1/N$ for all $i\in[\![1,N]\!]$,  this is the optimal uniform quantization problem to which Corollary \ref{cor:sparse_ot_bar} apply as well.

\section{Numerical experiments}\label{sec:simus}
We illustrate\footnote{The code to reproduce the experiments is available in Python at \url{https://github.com/LeoPtls}.} Corollary \ref{cor:sparse_ot_bar} for $\mathbb{D}=W_2^2$ and $W_{\epsilon,2}^2$.

\paragraph{Dependency in $n$} We consider  $L=3$ target Gaussian distributions in dimension $d=2$ with random means and covariances. We compute the $W_2^2$ and $W_{\epsilon,2}^2$ ($\epsilon = 1$) empirical barycenter for $N=10$, $N=50$ and for $n$ ranging between $1000$ and $10000$ using the POT library \cite{flamary2021pot}. Figure \ref{fig:convergence_in_n} represents the estimation error in \eqref{eq:cor:approx_error} as $n$ increases, matching the rate in Corollary \ref{cor:sparse_ot_bar}.
The presence of hollows in the curve for $\mathbb{D}=W_{\epsilon,p}^p$ is due to numerical instabilities.
\begin{figure}[H]
    \centering
        
    \subfloat[$\mathbb{D}=W_{2}^2$]{
    \includegraphics[width=0.45\linewidth]{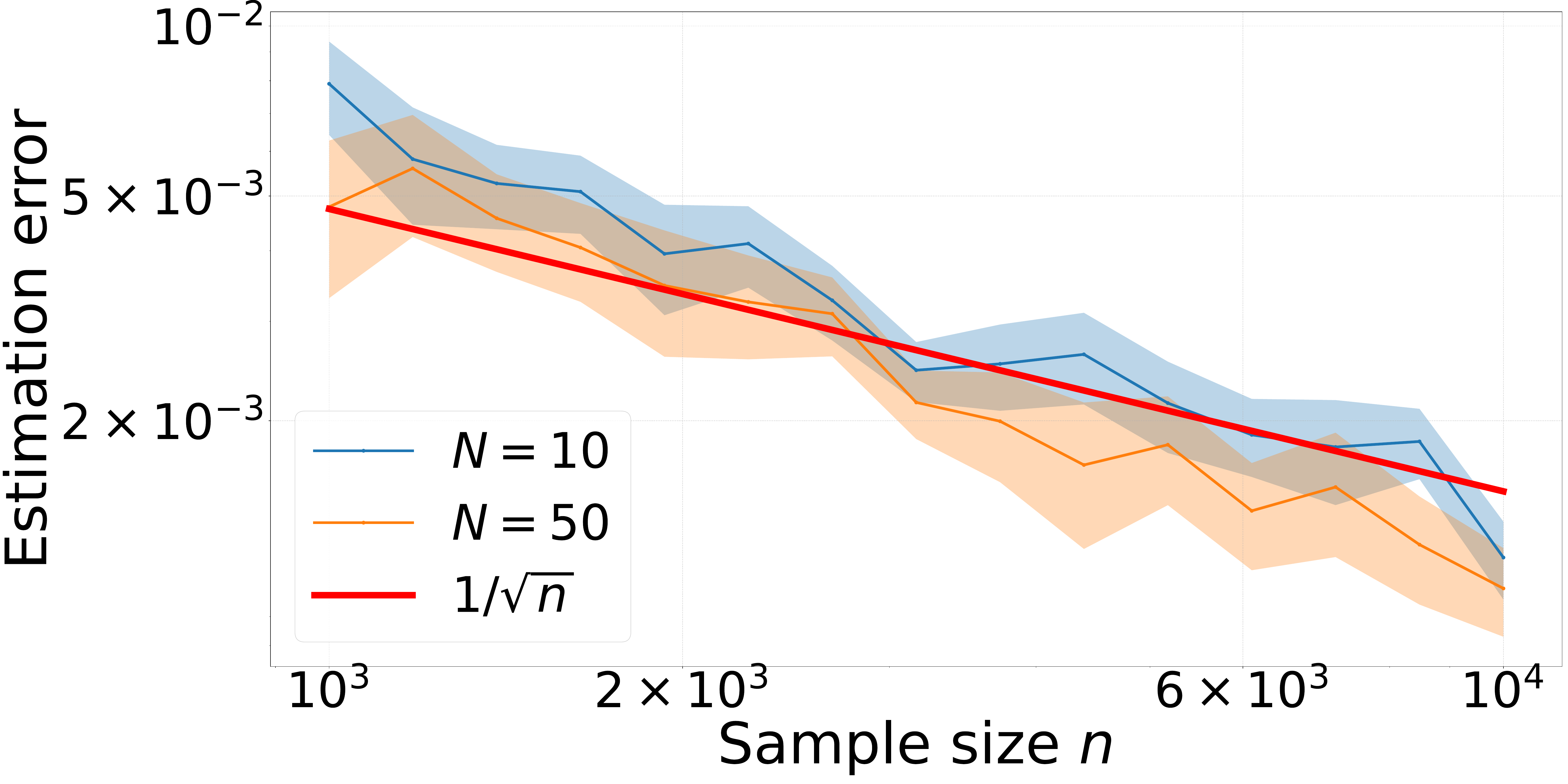}
    }
    \subfloat[$\mathbb{D}=W_{\epsilon,2}^2$]{
    \includegraphics[width=0.45\linewidth]{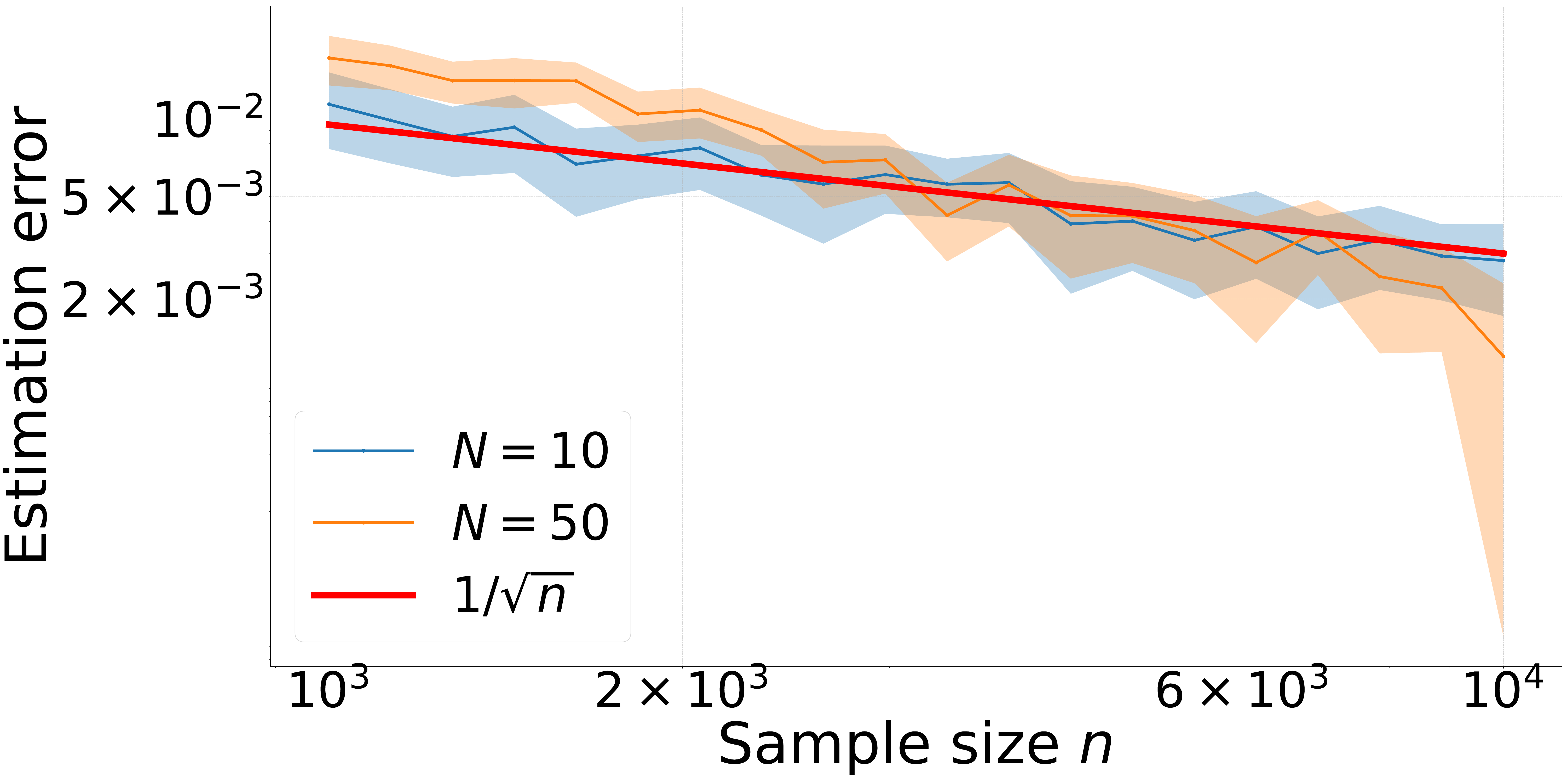}
    }
    \caption{The graphs show the evolution of the estimation error (in log-scale) as the sample size $n$ increases. Each experiments are performed $40$ times. The mean value of the estimation errors is plotted along with the $95\%$ confidence interval.  The theoretical upper bound $n\mapsto1/\sqrt{n}$ is plotted in red.}
    
    \label{fig:convergence_in_n}
\end{figure}
\paragraph{Dependency in $N$}
We consider the same setting as above in dimension $d=6$ for which the theoretical rates predict larger errors as $N$ increases. We fix the sample size $n=500$ and consider $N$ ranging between $10$ and $500$. Figure \ref{fig_Nvaries} shows the estimation error as a function of $N$ together with the upper bound of Corollary \ref{cor:sparse_ot_bar} and lower bound of Lemma \ref{Lem:lowerbound} for $\mathbb{D}=W_{2}^2$. This experiment suggests that the dependency in $N$ given by the worst case scenario in Lemma \ref{Lem:lowerbound} for $W_2^2$ matches practical observations on the mild gaussian distributions setting.

\begin{figure}[H]
    \centering
    \subfloat[$\mathbb{D}=W_2^2$]{%
        \includegraphics[width=0.45\linewidth]{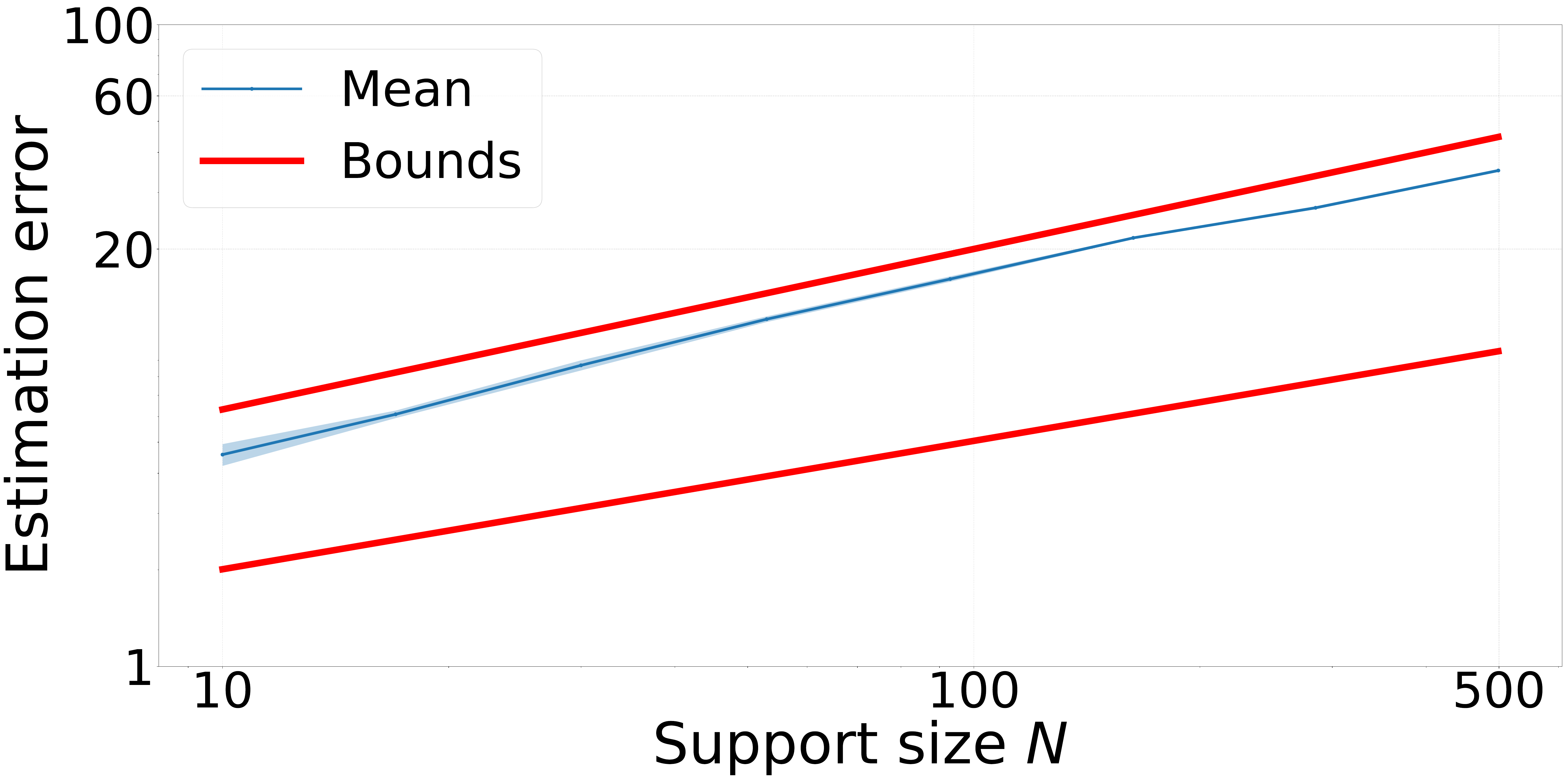}
    }
    \qquad
    \subfloat[$\mathbb{D}=W_{\epsilon,2}^2$]{%
        \includegraphics[width=0.45\linewidth]{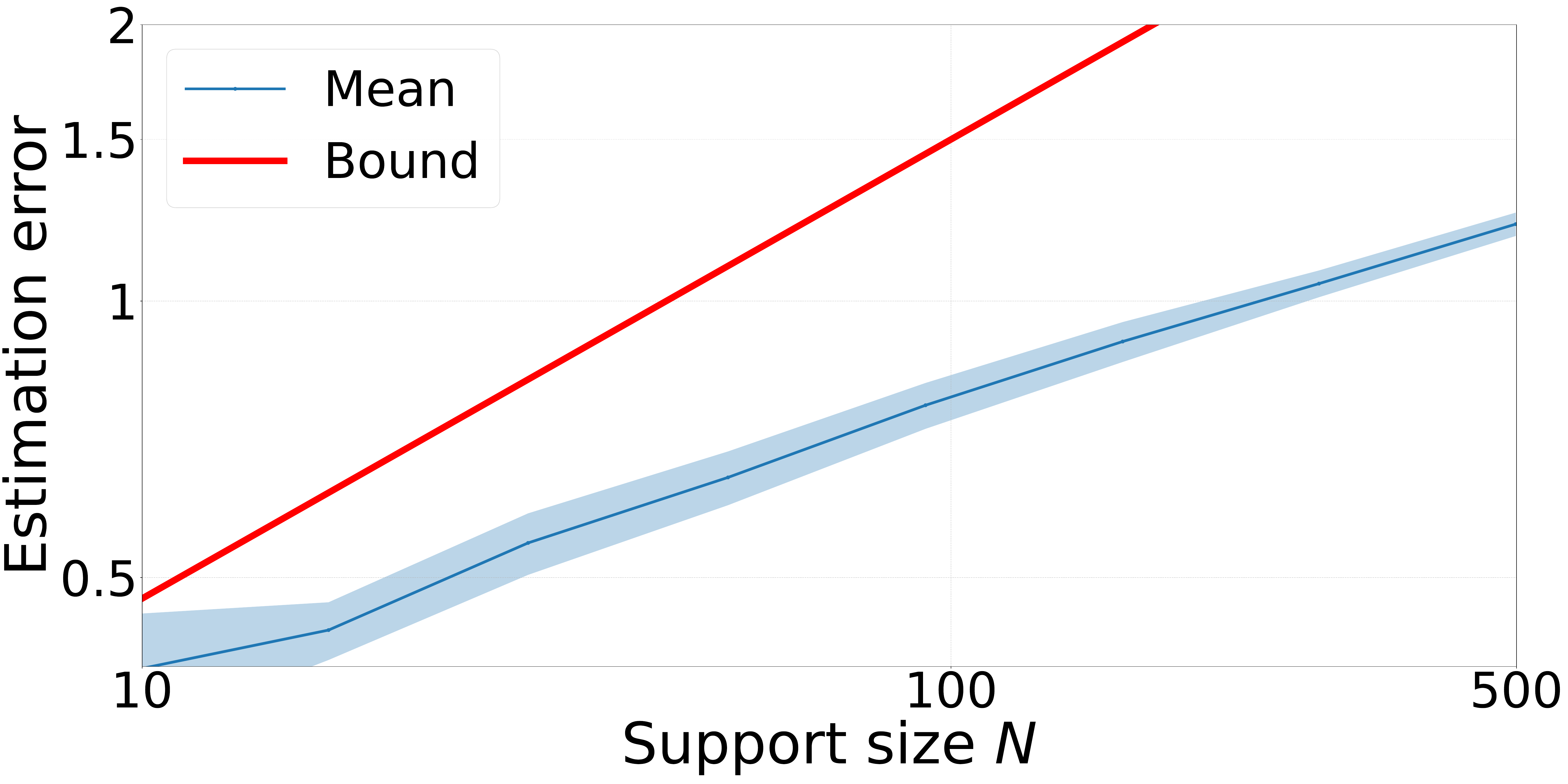}
    }
    \caption{The graphs show the estimation error (in log-scale) as the support cardinality $N$ increases, for $n=500$.  Each experiments are performed $10$ times. The mean value of the estimation errors is plotted along with the $95\%$ confidence interval. The theoretical bounds (Corollary \ref{cor:sparse_ot_bar} and Lemma \ref{Lem:lowerbound}) are plotted in red.}
    \label{fig_Nvaries}
\end{figure}

\section{Conclusion}
The expected convergence rate of the empirical sparse optimal transport barycenters is shown to be of order $O\left(\sqrt{N/n}\right)$ for several types of optimal transport divergences. The rate is min-max optimal for the Wasserstein distance $\mathbb{D}=W_2^2$. Investigating optimality for sliced distances $\mathbb{D}=SW_p^p,\:\text{max-}SW_p^p$ and Sinkhorn divergences $\mathbb{D}=W_{\epsilon,p}^p$ is left for future work. 
\
\appendix

\section{Semi-dual formulation of optimal transport}\label{Semi-dual formulation}

The Kantorovich dual formulation of the Wasserstein distances $W_p$ can be recast as a semi-dual problem using the $c$-transform. Similarly, the Sinkhorn divergences $W_{\epsilon,p}^p$ admit a semi-dual formulation based on the smoothed $c$-transform. In the semi-discrete setting (\textit{i.e.} between a continuous and a discrete measure), these formulations turn an infinite-dimensional program (corresponding to the primal formulations \eqref{eq:OT} and \eqref{eq:OT_eps}) into a finite-dimensional one, which is a key point in our proofs. In this section, we present these semi-dual formulations in the semi-discrete setting, with a particular focus on the case where the discrete measure has zero weights or repeated atoms. Throughout this section, we consider a measure $\mu\in \mathcal{M}_1(\R^d)$.

\subsection{Semi-dual formulation outside the generalized diagonal}\label{Appendix_semi-dual_outside_diagonal}

The generalized diagonal plays a central role throughout our proofs, it denotes the set of point clouds with at least two identical components:
\begin{equation}
D_N=\left\{Y:=(y_1,\dots y_N)\in(\R^d)^N\:|\:  \exists i\neq j\:\text{such that}\: y_i= y_j\: \right\}.
\end{equation}
Let us then define $\nu:=\sum_{i=1}^N\pi_i\delta_{y_i}$ a discrete measure supported on $Y:=(y_1,\ldots,y_N)\in (\R^d)^N\setminus D_N$ and strictly positive weights $\pi\in\mathrm{int}(\Delta_N)$. For a bounded continuous function $\phi\in \mathcal{C}_b(\R^d)$ , let us also denote $\phi^c\in \mathcal{C}_b(\R^d)$ its c-transform, defined in our case for all $x\in \R^d$ as $\phi^c(x)=\min_{i=1,...,N}(\|x-y_i\|^p-\phi(y_i))$ \cite[Section 1.3]{book_santambrogio}. Then by Kantorovich's duality (\cite[Theorem 5.10]{villani_opt_old_new}) and $c-$convexity of the optimal potentials we have the following formulation for the $p$-Wasserstein distance between $\mu$ and $\nu$:
\begin{equation}\label{eq:kantorovich_semi_dual}
W_p^p(\mu,\nu)=\underset{\phi\in \R^N}{\max}\int_{\R^d}\min_{i=1,\ldots,N}\{\|x-y_i\|^p-\phi_i\}d\mu(x)+\sum_{i=1}^N\pi_i\phi_i.
\end{equation}
In \eqref{eq:kantorovich_semi_dual}, with a slight abuse of notation, the dual potential $\phi \in \mathcal{C}_b(\R^d)$ is identified with its value on the support of $\nu$, which consists of $N$ distinct points: $\phi_i = \phi(y_i)$, for $i=1,\ldots,N$.
For this dual formulation to hold, the measure $\nu$ must be supported outside of the generalized diagonal $D_N$, with strictly positive weights; see the discussion in the next Section \ref{sec_app:diagonal}.

Entropic optimal transport \eqref{eq:OT_eps} between $\mu$ and $\nu$ also admits both a dual and a semi-dual form (see \cite[Proposition 2.1]{genevay2016stochastic}), which are respectively given by the following optimization problems:\begin{equation}\label{entropic_dual}
\begin{split}
W_{\epsilon,p}^p(\mu,\nu)   &=\underset{u,v\in \mathcal{C}(\R^d)\times \R^N}{\max}\int_{\mathbb{R}^d}u(x)d\mu(x)+\sum_{i=1}^N\pi_iv_i \\& \qquad- \epsilon\sum_{i=1}^N\pi_i\int_{\R^d}\exp\left(\frac{u(x)+v_i-\|x-y_i\|^p}{\epsilon} \right)d\mu(x)
\end{split}
\end{equation}
and 
\begin{equation}\label{entropic_semi_dual}
W_{\epsilon,p}^p(\mu,\nu)=\underset{w\in  \R^N}{\max}-\epsilon\int_{\R^d}\log\left(\sum_{i=1}^N\pi_i\exp\left(\frac{w_i-\|x-y_i\|^p}{\epsilon} \right) \right)d\mu(x) +\sum_{i=1}^N \pi_iw_i-\epsilon.
\end{equation}
These constructions also necessitate that $Y:=(y_1,\ldots,y_N)\in (\R^d)^N\setminus D_N$ and $\pi\in\mathrm{int}(\Delta_N)$.

\subsection{The case of support points on the generalized diagonal $D_N$ and/or with null weights}\label{sec_app:diagonal}
 As previously mentioned, it is mandatory that $(Y,\pi)\in \big(\R^d\big)^N\setminus D_N\times \mathrm{int}(\Delta_N)$ for a functional of the form $(Y,\pi)\mapsto \mathbb{D}\left(\mu,\sum_{i=1}^N\pi_i\delta_{y_i}\right)$ to admit the duality formulations described above.
 Indeed, suppose that $y_i=y_j$, for some $i \neq j$, then the dual formulation of the problem should include the constraint $\phi_i = \phi_j$, as both $\phi_i$ and $\phi_j$ represent the evaluation of a continuous potential in the same argument. In other words, we would have $\phi_i = \phi(y_i) = \phi(y_j) = \phi_j$.
While this may seem like a minor technicality that can often be overlooked, it cannot be omitted in our analysis. This is particularly relevant in the case of the sliced Wasserstein distance, where two distinct points in $\R^d$ may  project onto the same point on a line. Similarly, it is important to ensure that the weights of the discrete measure remain in the interior of the probability simplex $\mathrm{int}(\Delta_N)$.

This technicality has been handled in \cite[Section 5.1]{portales2024sequential} with the following proposition.
\begin{prop}\label{prop:nondegenerated_version}
Let $(Y,\pi)\in (\mathbb{R}^d)^N\times\Delta_N$, then there exists $M\in [\![1,N]\!]$ and $(\tilde{Y},\tilde{\pi})\in (\R^d)^M\setminus D_M\times \mathrm{int}(\Delta_M)$ such that for all $\mu\in\mathcal{M}_1(\R^d)$ and any $p\geq 1$:
\begin{equation*}
W_p\left(\mu,\sum_{i=1}^N\pi_i\delta_{y_i}\right) = W_p\left(\mu,\sum_{i=1}^M\tilde{\pi}_i\delta_{\tilde{y}_i}\right)\quad \text{and} \quad W_{\epsilon,p}\left(\mu,\sum_{i=1}^N\pi_i\delta_{y_i}\right)  = W_{\epsilon,p}\left(\mu,\sum_{i=1}^M\tilde{\pi}_i\delta_{\tilde{y}_i}\right).
\end{equation*}
\end{prop}

\subsection{Semi-duality formulation for semi-discrete sliced divergences}\label{sec:sliced_semi_dual}

Starting from the definition of the sliced distances in \eqref{eq:sliced_max}, the dual formulation \eqref{eq:kantorovich_semi_dual} can be applied to each one-dimensional distance $W_p(P_{\theta\sharp}\mu,P_{\theta\sharp}\nu),\:\theta\in\mathbb{S}^{d-1}$, addressing the ties in the support point of the measures $P_{\theta\sharp}\nu$. Let $(Y,\pi)\in(\R^d)^N\setminus D_N\times \mathrm{int}(\Delta_N)$, then, for each $\theta\in\mathbb{S}^{d-1}$ there exists an integer $M:=M(\theta)\in [\![1,N]\!]$ such that  $\langle y_i-y_j, \theta\rangle \neq 0$ and $\pi_i>0$ for all $i,j\in [\![1,M(\theta)]\!],\:i\neq j$. Up to a reordering of components, we can write $\pi_i$ as $\pi_i^\theta:=\pi_i+\pi_j$ whenever $\langle y_i-y_j,\theta\rangle=0$  for all $ i,j \in [\![1,N]\!], \:i\neq j$. Following this construction, we obtain a probability measure $\nu^\theta:=\sum_{i=1}^{M(\theta)}\pi_i^\theta\delta_{\langle y_i,\theta\rangle}$ with no ties in its support points. Then, the Kantorovich duality formulation \eqref{eq:kantorovich_semi_dual} leads to the following expressions of the sliced distances :

\begin{equation}\label{eq:sliced_dual}
SW_{p}^{p}(\mu,\nu)=\int_{\mathbb{S}^{d-1}}\left[\underset{w\in \R^{M(\theta)}}{\max}\int_{\mathbb{R}^{d}}\underset{i=1,\ldots, M(\theta)}{\min}\left\{ \vert\langle x-y_{i},\theta \rangle\vert^{p}-w_{i} \right\}d\mu(x) + \sum_{i=1}^{M(\theta)} \pi_i w_{i}\right]d\sigma(\theta)
\end{equation}
and
\begin{equation}\label{eq:maxsliced_dual}
\text{max-}SW_p^p(\mu,\nu)=\underset{\theta\in \mathbb{S}^{d-1}}{\max}\underset{w\in \R^{M(\theta)}}{\max}\int_{\mathbb{R}^{d}}\underset{i=1,\ldots, M(\theta)}{\min}\left\{ \vert \langle x-y_{i},\theta \rangle\vert^{p}-w_{i} \right\}d\mu(x) + \sum_{i=1}^{M(\theta)} \pi_i w_{i},
\end{equation}
where $\sigma$ denotes the uniform measure over the $d$-sphere $\mathbb{S}^{d-1}$. Note that actually, one can replace $M(\theta)$ by $N$ in the maximization over the variable $w$ in \eqref{eq:sliced_dual}, without taking into account any dependency on $\theta$. Indeed, let
$M(\theta) := \mathrm{card}\left( \left\{ i \in [\![1, N]\!] \,:\, \langle y_i - y_j, \theta \rangle \neq 0, \:\forall j\neq i\right\} \right).$
We thus have $M(\theta)<N$ if and only if there exist indices $i \neq j$ such that $ \langle y_i - y_j, \theta \rangle = 0$, meaning that $\theta$ lies in a union of sets of the form $\left\{ \theta \in \mathbb{S}^{d-1} \,:\, \langle y_i - y_j, \theta \rangle = 0 \right\},$ for $\:i,j\in[\![1,N]\!]$ and $i\neq j$. 
Since $ Y \in (\mathbb{R}^d)^N \setminus D_N$ this is a finite union of hyperplanes intersected with the sphere, which is $\sigma$-negligible. It follows that $M(\theta) = N$  for almost every $ \theta \in \mathbb{S}^{d-1} $.

\section{Boundedness of the dual variables}\label{Appenxid_boundedness}
In this section, we consider a measure $\mu\in \mathcal{M}_1(\Br),\:R>0$ and a discrete measure $\nu\in\mathcal{M}_1^N(\Br)$. We prove that the dual variables associated to the optimal transport problem $\mathbb{D}\left(\mu,\nu\right)$ for the Wasserstein distances $\mathbb{D}=W_p^p$ and the Sinkhorn divergences $\mathbb{D}=W_{\epsilon,p}^p$ are bounded.

\begin{prop}\label{Prop:barycentersproperties}
Let $\mu\in \mathcal{M}_1(\Br)$ for some $R>0$ and $\mathbb{D}=W_p^p,\:\text{or}\:W_{\epsilon,p}^p$. Then for any $(Y,\pi)\in (\Br^N\setminus D_N)\times \mathrm{int}(\Delta_N)$, we have 
\begin{equation*}
W_p^p\left(\mu,\sum_{i=1}^N\pi_i\delta_{y_i}\right)=\underset{\mid w_i \mid\leq K_1}{\underset{w \in \R^N}{\max}}\int_{\R^d}\min_{i=1,\ldots,N}\{\|x-y_i\|^p-w_i\}d\mu(x)+\sum_{i=1}^N\pi_iw_i,
\end{equation*}
and
\begin{equation*}
W_{\epsilon,p}^p\left(\mu,\sum_{i=1}^N\pi_i\delta_{y_i}\right)=\underset{\mid w_i\mid\leq K_2}{\underset{w\in  \R^N}{\max}}-\epsilon\int_{\R^d}\log\left(\sum_{i=1}^N\pi_i\exp\left[\frac{w_i-\|x-y_i\|^p}{\epsilon} \right] \right)d\mu(x) +\sum_{i=1}^N \pi_iw_i-\epsilon
\end{equation*}
where $K_1=2pR^p$ and $K_2=4p(2R)^p$.
\end{prop}

\begin{rem}\label{rem:dual_cst}
In all generality, the dual variables $w\in\R^N$ are not bounded since adding a constant $c\in\R$ to $w$ yields another valid solution. Therefore, to fix a dual solution, we set one coordinate of  $w$ to $0$ and under this constraint the dual variable $w$ remains bounded uniformly in $Y$ and $\pi$.
\end{rem}

\begin{proof}[Proof of Proposition \ref{Prop:barycentersproperties}]  In the following, we prove both cases independently.
\begin{description}
\item[Case $\mathbb{D}=W_{p}^p$.] 

Let $(Y,\pi)\in (\Br^N\setminus D_N)\times \mathrm{int}(\Delta_N)$, then by Kantorovich duality \eqref{eq:kantorovich_semi_dual} between $\mu$ and the discrete measure $\sum_{i=1}^N \pi_i \delta_{y_i}$, we have
\begin{equation}\label{duality_proof_prop:barycenter:borné}
W_p^p\left(\mu,\sum_{i=1}^N \pi_i \delta_{y_i} \right)=\max_{w\in \R^N} \int_{\R^d}\underset{i=1,\ldots,N}{\min}\{\|x-y_i\|^p-w_i\}d\mu(x)+\sum_{i=1}^N\pi_iw_i.
\end{equation} 
As discussed in Remark \eqref{rem:dual_cst}, we may select a choice of dual variables $w\in \R^N$ optimal in \eqref{duality_proof_prop:barycenter:borné} and verifying $w_N=0$. From \cite[Proposition 5.11]{book_santambrogio} and \cite[Remark 1.13]{book_santambrogio} that follows, we may also select a c-convave function $w^c$ which verifies for all $i\in [\![1,N]\!]$:  $w_i=(w^c(y_i))^c$. We then have 
\begin{align}
\vert w_i-w_N \vert &=\vert (w^c(y_i))^c-(w^c(y_N))^c \vert\nonumber\\
&=\left| \min_{x\in \supp(\mu)}\left\{\|x-y_i\|^p-w^{c}(x)\right\} - \min_{x\in \supp(\mu)}\left\{\|x-y_N\|^p-w^{c}(x)\right\}\right|\nonumber\\
&\leq \max_{x\in\Br}\vert \|x-y_i\|^p- \|x-y_N\|^p\vert\label{ineq_1}\\
&\leq pR^{p-1}\|z_i-y_i\| \leq 2pR^p,\label{ineq_2}
\end{align}
 where inequalities \eqref{ineq_1} and \eqref{ineq_2} come from Lemmas \ref{lem:ineq_sup} and \ref{lem:TAF-cout-norme}. Using that $w_N=0$, we get the result: the dual variables for the $p$-Wasserstein distances are bounded, and $\vert w_i\vert \leq 2pR^p$.

\item[Case $\mathbb{D}=W_{\epsilon,p}^p$.]

Let  $(Y,\pi)\in(\Br^N\setminus D_N)\times \mathrm{int}(\Delta_N)$. Let also $u\in \mathcal{C}_b(\Br)$ be an optimal solution in the dual \eqref{entropic_dual} formulation of the entropic $p$-Wasserstein distance between $\mu$ and $\sum_{i=1}^N\pi_i\delta_{y_i}$. By \cite[Proposition 1]{genevay2019sample}, $u$ is $\kappa$-Lipshitz, where $\kappa$ is the Lipschitz constant of $(a,b)\in\Br^2\mapsto\|a-b\|^p$. Then since $\supp(\mu)\subset \Br$, we have by Lemma \ref{lem:TAF-cout-norme} that $\kappa$ is upper bounded by $p\sqrt{2}(2R)^{p-1}$. 
Let $w$ be a solution of \eqref{entropic_semi_dual} for which we can suppose $w_N=0$ (Remark \eqref{rem:dual_cst}). As in \cite{genevay2016stochastic}, one can show with a first order condition on \eqref{entropic_dual} that for all $y$ a coordinate in $Y$: 
$w(y)=-\epsilon\log\left(\int_{\R^d} \exp\left(\frac{u(x)-\|x-y\|^p}{\epsilon}\right)d\mu(x)\right).$
We thus have for all $i\neq N$:
\begin{align}
\vert w(y_i)-w(y_N)\vert&=\left\vert \epsilon\log \left(\frac{\int_{\R^d}\exp\left(\frac{u(x)-\|x-y_i\|^p}{\epsilon}\right)d\mu(x)}{\int_{\R^d}\exp\left(\frac{u(x)-\|x-y_N\|^p}{\epsilon}\right)d\mu(x)} \right)\right\vert\nonumber\\
&\leq \left\vert\epsilon\log \left(\frac{\int_{\R^d}\exp\left(\frac{u(\overline{x}_i)-\|\overline{x}_i-y_i\|^p}{\epsilon}\right)d\mu(x)}{\int_{\R^d}\exp\left(\frac{u(\underline{x}_N)-\|\underline{x}_N-y_N\|^p}{\epsilon}\right)d\mu(x)} \right)\right\vert\nonumber\\
&= \vert u(\overline{x}_i)-\|\overline{x}_i-y_i\|^p-(u(\underline{x}_N)-\|\underline{x}_N-y_N\|^p) \vert\nonumber\\
&\leq \vert u(\overline{x}_i)-u(\underline{x}_N) \vert+\vert \ \|\overline{x}_i-y_i\|^p-\|\underline{x}_N-y_N\|^p \vert \nonumber\\
&\leq p\sqrt{2}(2R)^{p-1}\|\overline{x}_i-\underline{x}_N\|+\|\overline{x}_i-y_i\|^p +\|\underline{x}_N-y_N\|^p\label{ineq_lem_boundedness_mvt_entropic}\\
&\leq p\sqrt{2}(2R)^{p} + 2(2R)^p  \nonumber\leq 4p(2R)^p,\nonumber
\end{align}
where $\overline{x}_i$ and $\underline{x}_N$ denote respectively a maximizer of $x\mapsto\exp\left(\frac{u(x)-\|x-y_i\|^p}{\epsilon}\right)$ and a minimizer of $x\mapsto\exp\left(\frac{u(x)-\|x-y_N\|^p}{\epsilon}\right)$ over $\mathbb{B}_R$. Using that $w(y_N)=0$ then yields the result.

\end{description}
\end{proof}

\section{Proofs of the main results}\label{Main_th_proof_appendix}
In this section, we prove the sample complexity of optimal transport barycenters with sparse support in Theorem \ref{th:main}. For $\ell \in[\![1,L]\!]$, we denote $\mathbb{P}_n^\ell$ the empirical process $
    f\mapsto \mathbb{P}_n^\ell f:= \frac{1}{n}\sum_{i=1}^nf(X_i^\ell)$
 for $n$ i.i.d random variables $X_1^\ell,\ldots, X_n^\ell$ of law $\mu^\ell$. For any subset $\mathcal{F}$ of a normed space $(\Theta,\Vert\cdot\Vert)$ and $\tau>0$, we denote $\mathcal{N}(\tau,\mathcal{F},\|\cdot\|)$  the $\tau$-covering number of $\mathcal{F}$ with respect to $\|\cdot\|$, which is the minimum number of balls of radius $\tau$, with respect to the norm $\|\cdot\|$, needed to cover $\mathcal{F}$. We refer the reader to \cite[Definition 2.2]{sen2018gentle} for more details on the covering number. Finally, we also denote $\Vert\cdot\Vert_{\mathbb{L}^2(\mu)}$ the $\mathbb{L}^2$ norm with respect to the measure $\mu$ and $\overline{\text{B}_d(0,R)}$ the $d$-dimensional closed ball of radius $R$.
The proofs for Wasserstein (Section \ref{sec:app_wass_barycenters}), Sliced Wasserstein (Section \ref{sec:app_sliced}) and Sinkhorn (Section \ref{sec:app_entropy}) barycenters all follow the same approach : we first derive the covering number for a class of functions associated to the dual formulation of the divergences $\mathbb{D}$, and then apply tools from empirical process theory to obtain a generalization error bound for the barycenters.

\subsection{Wasserstein barycenter}\label{sec:app_wass_barycenters}
\begin{lem}\label{lem:covering_number_F_p}
Let 
\begin{equation}\label{eq:functional_class_W_p}
\mathcal{F}_p=\left\{\left\{f_Y^w:x\in\Br\mapsto \underset{i=1,...,N}{\min}\|x-y_i\|^p-w_i\right\},\: Y\in \Br^N,\:w\in  \overline{\text{B}_1(0,2pR^p)}^N\right\},
\end{equation}
then for all $\tau>0$ and $\mu\in \mathcal{M}_1(\Br)$,
$\mathcal{N}(\tau,\mathcal{F}_p,\Vert\cdot\Vert_{\mathbb{L}^2(\mu)})\leq\left(1+\frac{4pR^p}{\tau}\right)^{dN}\left(1+\frac{8pR^p}{\tau}\right)^N$.
Also, the subset $\mathcal{F}_{p,\mathbb{Q}}$ of functions of $\mathcal{F}_p$ with rationals parameters $(Y,w)$, is dense in $\mathcal{F}_p$.
\end{lem}
\begin{proof}

Let $(Y,w),\:(Z,v)\in \Br^N\times [-2pR^p,2pR^p]^N $. Then for all $x\in \Br$ 
\begin{align}
\vert f_Y^w(x)-f_{Z,v}(x)\vert&=\left\vert \underset{i=1,...,N}{\min}\|x-y_i\|^p-w_i - \underset{i=1,...,N}{\min}\|x-z_i\|^p-v_i \right\vert \nonumber\\
&\leq \underset{i=1,...,N}{\max}\vert \|x-y_i\|^p-w_i- (\|x-z_i\|^p-v_i) \vert \label{inequality_max}\\
&\leq pR^{p-1}\underset{i=1,...,N}{\max}\|y_i-z_i\|+\underset{i=1,...,N}{\max}\vert w_i-v_i\vert\label{MVT},
\end{align}
where inequalities \eqref{inequality_max} and \eqref{MVT} are due to Lemmas \ref{lem:ineq_sup} and \ref{lem:TAF-cout-norme}.
Let us denote $T=pR^{p-1}$, and let $\tau>0$. We then consider $z^1,\ldots,z^{n(\tau)}$ a $\frac{\tau}{2T}$-covering of $\Br$ and $v^1,\ldots v^{m(\tau)}$ a $\frac{\tau}{2}$-covering of $\overline{\text{B}_1(0,2pR^p)}$. For any vector of indices $i=(i_1,\ldots,i_N)\in [\![1,n(\tau)]\!]^N$ and $k=(k_1,\ldots,k_N)\in [\![1,m(\tau)]\!]^N$, we define for all $x\in \Br$ the function $f_{i,k}(x)=\underset{j=1,...,N}{\min}\|x-z^{i_j}\|^p-v^ {k_j}$. Let $f:=f_Y^w\in\mathcal{F}_p$, then by \eqref{MVT} there exists $i,k$ such that
\begin{equation*}
\vert f(x)-f_{i,k}(x)\vert \leq pR^{p-1}\max_{j=1,...,N}\|y_j-z^{i_j}\|+\max_{j=1,...,N}\vert w_j-v^{k_j}\vert\leq T\frac{\tau}{2T}+\frac{\tau}{2}=\tau,
\end{equation*}
so that $\|f-f_{i,k}\|_{\mathbb{L}^2(\mu)}\leq \tau$.
Since there exists $n(\tau)^N\times m(\tau)^N$ possibilities for such functions $f_{i,k}$, it follows from the covering construction that :
\begin{equation*}
\begin{split}
\mathcal{N}(\tau,\mathcal{F}_p,\|\cdot\|_{\mathbb{L}^2(\mu)})
&\leq \mathcal{N}\left(\frac{\tau}{2T},\overline{\text{B}_d(0,R)},\|\cdot\|\right)^N\mathcal{N}\left(\frac{\tau}{2},\overline{\text{B}_1(0,2pR^p)},\mid\cdot\mid\right)^N\\
&\leq \left(1+\frac{4RT}{\tau}\right)^{dN}\left(1+\frac{8pR^p}{\tau}\right)^{N}=\left(1+\frac{4pR^p}{\tau}\right)^{dN}\left(1+\frac{8pR^p}{\tau}\right)^N
\end{split}
\end{equation*}
where the second inequality is due to \cite[Lemma 4.14]{massart2007concentration}.

Finally, $\mathcal{F}_{p,\mathbb{Q}}$ is dense in  $\mathcal{F}_{p}$ by a continuity argument and inequality \eqref{MVT}.
\end{proof}

\begin{proof}[Proof of Theorem \ref{th:main} for $\mathbb{D}=W_p^p$]
Let $(Y,\pi)\in \Br^N\times \Delta_N$ and let $M\in[\![1,N]\!]$ and $(\tilde{Y},\tilde{\pi})\in(\Br^M\setminus D_M)\times \mathrm{int}(\Delta_M)$ be as in Proposition \ref{prop:nondegenerated_version}. 
For all $\ell\in [\![1,L]\!]$, let $X_1^\ell,\ldots,X_n^\ell$  be $n$ i.i.d random variables of law $\mu^\ell$. Using Proposition \ref{Prop:barycentersproperties} and denoting $K:=2pR^p$ we therefore have
\begin{equation*}
\begin{split}
&\text{Gen}_n(Y,\pi):=\left| \bary_\mathbb{D}\left(\mu_n^1,\ldots,\mu_n^L,\sum_{i=1}^N\pi_i\delta_{y_i}\right) -\bary_\mathbb{D}\left(\mu^1,\ldots,\mu^L,\sum_{i=1}^N\pi_i\delta_{y_i}\right)\right|\\
&\leq \frac{1}{L}\sum_{\ell=1}^L\underset{\mid \tilde{w}_i\mid \leq K}{\underset{\tilde{w}\in \R^M}{\max}}\left| \frac{1}{n}\sum_{k=1}^n\underset{i=1,\ldots,M}{\min}\{\|X_k^\ell-\tilde{y_i}\|^p-\tilde{w}_i\}-\int_{\R^d} \underset{i=1,\ldots,M}{\min}\{\|x-\tilde{y}_i\|^p-\tilde{w}_i\}d\mu^\ell(x)\right|\\
&=\frac{1}{L}\sum_{\ell=1}^L\underset{w_i=w_j\:\text{if}\:y_i=y_j}{\underset{\mid w_i\mid \leq K}{\underset{w\in \R^N}{\max}}}\left| \frac{1}{n}\sum_{k=1}^n\underset{i=1,\ldots,N}{\min}\{\|X_k^\ell-y_i\|^p-w_i\}-\int_{\R^d} \underset{i=1,\ldots,N}{\min}\{\|x-y_i\|^p-w_i\}d\mu^\ell(x)\right|.
\end{split}
\end{equation*}
where the inequality is due to Lemma \ref{lem:ineq_sup}. The last equality is due to the fact that for all $Y\in \Br^N$  and $x\in \R^d$, we have $\min_{k=1,\ldots,M}\|x-y_k\|^p-w_k=\min_{k=1,\ldots,N}\|x-y_k\|^p-w_k$ if $w_i=w_j$ whenever $y_i=y_j$.
Notice that the dependency in the vector of weights $\pi$ has disappeared.
For any $(x,Y,w)\in \Br\times \Br^N \times [-K,K]^N$, let $f_Y^w(x)=\underset{i=1,\ldots,N}{\min}\|x-y_i\|^p-w_i$. We then have
\begin{equation*}
\begin{split}
\underset{(Y,\pi)\in \Br^N\times \Delta_N}{\sup} \text{Gen}_n(Y,\pi)
&\leq \frac{1}{L}\sum_{\ell=1}^L\underset{(Y,w)\in \Br^N \times[-K,K]^N}{\sup}\left| \mathbb{P}_n^\ell\left( f_Y^w \right)-\mathbb{E}_{X\sim \mu^\ell}\left[ f_Y^w(X)\right]\right|\\
&=\frac{1}{L}\sum_{\ell=1}^L\underset{f\in \mathcal{F}_p}{\sup} \vert\mathbb{P}_n^\ell(f)-\mathbb{E}_{X\sim \mu^\ell}[f(X)]\vert
\end{split}
\end{equation*}
where $\mathcal{F}_p$ is defined in \eqref{eq:functional_class_W_p}.
Notice further that for any $\ell\in[\![1,L]\!]$,
\begin{equation*}
\begin{split}
\sup_{f\in\mathcal{F}_p}\mathbb{P}^\ell_nf^2&=\sup_{(Y,w)\in \Br^N\times [-K,K]^N}\frac{1}{n}\sum_{k=1}^n \left(\min_{i=1,...,N}\|X_k^\ell-y_i\|^p-w_i\right)^2\\
&\leq \frac{1}{n}\sum_{k=1}^n\left(\sup_{x,y\in\Br}\|x-y\|^p+\sup_{w\in[-K,K]} \vert w\vert \right)^2\\
&\leq ((2R)^p+2pR^p)^2\leq(2(2R)^p)^2,
\end{split}
\end{equation*}
which gives $\sqrt{\sup_{f\in\mathcal{F}}\mathbb{P}_nf^2}\leq2(2R)^p$. By Lemma \ref{lem:covering_number_F_p}, $\mathcal{F}_p$ is dense in $\mathcal{F}_{p,\mathbb{Q}}$, and we can then apply \cite[Theorem 3.5.1]{gine2021mathematical} on empirical processes and covering number. Therefore for all $\ell$: 
\begin{equation*}
\mathbb{E}\left[\sqrt{n}\underset{f\in \mathcal{F}_p}{\sup} \vert\mathbb{P}_n^\ell(f)-\mathbb{E}_{X\sim \mu^\ell}[f(X)]\vert\right]
\leq 8\sqrt{2} \mathbb{E}\left[\int_{0}^{2(2R)^p}\sqrt{\log\left(2\mathcal{N}(\tau/2,\mathcal{F}_p,\Vert \cdot\Vert_{\mathbb{L}^2(\mu_n^\ell)})\right)}d\tau\right].
\end{equation*}
From Lemma \ref{lem:covering_number_F_p} applied to the measure $\mu_n^\ell$, denoting $\tilde{K} = 2(2R)^p$, we get
\begin{equation*}
\begin{split}
\int_{0}^{\tilde{K}}\sqrt{\log(2\mathcal{N}(\tau/2,\mathcal{F}_p,\|\cdot\|_{\mathbb{L}^2(\mu_n^\ell)}))}d\tau
&\leq \int_{0}^{\tilde{K}}\sqrt{\log\left(2 \left[1+\frac{8pR^p}{\tau}\right]^{dN}\left[1+\frac{16pR^p}{\tau}\right]^N \right)}d\tau\\
&\leq\sqrt{N(d+1)}\int_{0}^{\tilde{K}}\sqrt{\log\left(2+\frac{32pR^p}{\tau}\right) }d\tau.
\end{split}
\end{equation*}
Finally, as this last term does not depend on $\ell$, we get
\begin{equation*}
\mathbb{E}\left[\frac{1}{L}\sum_{\ell=1}^L\underset{f\in \mathcal{F}_p}{\sup} \vert\mathbb{P}_n^\ell(f)-\mathbb{E}_{X\sim \mu^\ell}[f(X)]\vert\right]\leq C_{p,R}\sqrt{\frac{N(d+1)}{n}},
\end{equation*}
where $C_{p,R}=8\sqrt{2}\int_{0}^{2(2R)^p}\sqrt{\log\left(2+\frac{32pR^p}{\tau}\right)}d\tau$.
\end{proof}

\subsection{Sliced Wasserstein barycenter}\label{sec:app_sliced}
\begin{lem}\label{lem:covering_number_F_p^s}
Let
\begin{equation*}
\mathcal{F}_p^s=\left\{\left\{f_{\theta,Y,w}:x\mapsto \underset{i=1,...,N}{\min}\vert \langle x-y_i,\theta  \rangle \vert^p-w_i\right\},\: Y\in \Br^N,\:w\in\overline{\text{B}_1(0,2pR^p)}^N,\theta\in \mathbb{S}^{d-1} \right\},
\end{equation*}
then for all $\tau>0$, $\mu\in \mathcal{M}_1(\Br)$,
$\mathcal{N}(\tau,\mathcal{F}_p^s,\|\cdot\|_{\mathbb{L}^2(\mu)})\leq \left(1+\frac{12pR^p}{\tau}\right)^d \left(1+\frac{6pR^{p}}{\tau}\right)^{dN} \left(\frac{12pR^p}{\tau}\right)^N.$
Also, the subset $\mathcal{F}_{p,\mathbb{Q}}^\text{s}$ of functions of $\mathcal{F}_p^\text{s}$ with rationals parameters $(Y,w,\theta)$, is dense in $\mathcal{F}_p^\text{s}$.
\end{lem}
\begin{proof}
Let $K=2pR^p$, $Y,Z\in \Br^N$, $w,v\in [-K,K]^N$ and $\theta,\varphi\in \mathbb{S}^{d-1}$. Then by Lemma \eqref{lem:TAF-cout-norme}, we have for $x\in\Br$
\begin{equation*}
\begin{split}
\left| \ |\langle x-y_i,\theta \rangle|^p-|\langle x-z_i,\varphi \rangle|^p\right|&\leq pR^{p-1}| \langle x-y_i,\theta \rangle-\langle x-z_i,\varphi \rangle  |\\
&= pR^{p-1}| \langle x,\theta-\varphi\rangle+ \langle z_i,\varphi-\theta\rangle +\langle z_i,\theta \rangle-\langle y_i,\theta \rangle |\\
&\leq pR^{p-1}(2R\|\theta-\varphi\|+\|z_i-y_i\|).
\end{split}
\end{equation*}
Therefore, thanks to Lemma \eqref{lem:ineq_sup}, we get 
\begin{align}
| f_{\theta,Y,w}(x)-f_{\varphi,Z,v}(x)|&\leq \max_{i=1,\ldots, N}  \ \left| \ |\langle x-y_i,\theta \rangle |^p-  | \langle x-z_i,\varphi \rangle |^p  \right|+| w_i-v_i|\nonumber\\
&\leq \max_{i=1,\ldots,N}\ pR^{p-1}(2R\|\theta-\varphi\|+\|z_i-y_i\|)+| w_i-v_i|\nonumber \\
&\leq 2pR^{p} \|\theta-\varphi\| + pR^{p-1}\max_{i=1,\ldots,N}\|z_i-y_i\|+\max_{i=1,\ldots,N}\vert w_i-v_i\vert .\label{ineq_MVT_sliced}
\end{align}
 Then, using the same argument as in the proof of Lemma \ref{lem:covering_number_F_p}, we can upper bound the covering number $\mathcal{N}(\tau,\mathcal{F}_p^s,\|\cdot\|_{\mathbb{L}^2(\mu)})$ by the following product of covering numbers:
\begin{equation*}
\begin{split}
\mathcal{N}(\tau,\mathcal{F}_p^s,\|\cdot\|_{\mathbb{L}^2(\mu)})&\leq \mathcal{N}\left[\frac{\tau}{6pR^p},\mathbb{S}^{d-1},\|\cdot\|\right]\mathcal{N}\left[\frac{\tau}{3pR^{p-1}},\mathbb{B}_R,\|\cdot\|\right]^N\mathcal{N}\left[\frac{\tau}{3}, \overline{\text{B}_1(0,2pR^p)},\vert\cdot\vert\right]^N\\
&\leq \left(1+\frac{12pR^p}{\tau}\right)^d \left(1+\frac{6pR^{p}}{\tau}\right)^{dN} \left(\frac{12pR^p}{\tau}\right)^N,\\
\end{split}
\end{equation*}
which concludes the proof of the upper bound of the covering number. Finally, $\mathcal{F}_{p,\mathbb{Q}}^s$ is dense in  $\mathcal{F}_{p}^s$ due to inequality \eqref{ineq_MVT_sliced}.
\end{proof}

\begin{proof}[Proof of Theorem \ref{th:main} for $\mathbb{D}=SW_p^p$ and  $\text{max-}SW_p^p$]
For brevity, we only prove the result for $\mathbb{D}=SW_p^p$ as the $\text{max-}SW_p^p$ case is similar. Let $M\in[\![1,N]\!]$ and $(\tilde{Y},\tilde{\pi})\in (\Br^M\setminus D_M)\times \mathrm{int}(\Delta_N)$ be as in Proposition \ref{prop:nondegenerated_version}. For all $\ell\in [\![1,L]\!]$, let $X_1^\ell,\ldots,X_n^\ell$ be i.i.d random variables of law $\mu^\ell$. Following the construction in Section \ref{sec:sliced_semi_dual}, we define the function $M:\theta\mapsto \card(\{i\in[\![1,M]\!],\:\langle y_i-y_j,\theta\rangle\neq 0, \forall j\neq i\})$. By Proposition \ref{Prop:barycentersproperties} and for $K:=2pR^p$, we get
\begin{equation*}
\begin{split}
&\text{Gen}_n^\text{s}(Y,\pi):=\left| \bary_\mathbb{D}\left(\mu_n^1,\ldots,\mu_n^L,\sum_{i=1}^N\tilde{\pi}_i\delta_{\tilde{y_i}}\right) -\bary_\mathbb{D}\left(\mu^1,\ldots,\mu^L,\sum_{i=1}^N\tilde{\pi}_i\delta_{\tilde{y_i}}\right)\right|\\
&\leq \frac{1}{L}\sum_{\ell=1}^L\underset{\theta \in \mathbb{S}^{d-1}}{\max}{\underset{\vert w_i\vert \leq K}{\underset{w\in \R^{M(\theta)}}{\max}}}\left| \frac{1}{n}\sum_{k=1}^n\underset{i=1,\ldots,M(\theta)}{\min}\{\vert \langle X_k^\ell-y_i,\theta \rangle \vert ^p-w_i\}\right.\\
&\left.-\int_{\R^d} \underset{i=1,\ldots,M(\theta)}{\min}\{\vert \langle x-y_i,\theta \rangle \vert ^p-w_i\}d\mu^\ell(x)\right|.\\
&= \frac{1}{L}\sum_{\ell=1}^L\underset{\theta \in \mathbb{S}^{d-1}}{\max}\underset{w_i=w_j\:\text{if}\:\langle y_i-y_j,\theta\rangle=0}{\underset{\vert w_i\vert \leq K}{\underset{w\in \R^N}{\max}}}\left| \frac{1}{n}\sum_{k=1}^n\underset{i=1,\ldots,N}{\min}\{\vert \langle X_k^\ell-y_i,\theta \rangle \vert ^p-w_i\}\right.\\
&\left.-\int_{\R^d} \underset{i=1,\ldots,N}{\min}\{\vert \langle x-y_i,\theta \rangle \vert ^p-w_i\}d\mu^\ell(x)\right|,
\end{split}
\end{equation*}
where the first inequality uses Lemma \ref{lem:ineq_sup} for the maximum with respect to $w$, and then upper bound the integral with respect to $\theta$ by its maximum value. The last equality is due to the fact that for all $\theta\in \mathbb{S}^{d-1}$, $Y\in (\R^d)^N$ and $x\in \R^d$, we have that
$\min_{k=1,\ldots, M(\theta)} \mid \langle x-y_k,\theta \rangle\mid^p-w_k=\min_{k=1,\ldots,N}\mid \langle x-y_k,\theta \rangle\mid^p-w_k$
whenever $w_i=w_j$ if $\langle y_i-y_j,\theta \rangle=0$.

For all $(x,\theta,Y,w)\in \Br\times \mathbb{S}^{d-1}\times \Br^N \times [-K,K]^N$, let $f_{\theta,Y,w}(x)=\underset{i=1,\ldots,N}{\min}\vert \langle x-y_i,\theta\rangle \vert^p-w_i$. We then have 
\begin{equation*}
\begin{split}
\underset{(Y,\pi)\in \Br^N\times \Delta_N}{\sup}\text{Gen}_n^\text{s}(Y,\pi)&\leq \frac{1}{L}\sum_{\ell=1}^L\ \underset{(\theta,Y,w)\in \mathbb{S}^{d-1}\times\Br^N \times[-K,K]^N}{\sup}\left| \mathbb{P}_n^\ell\left( f_{\theta,Y,w} \right)-\mathbb{E}_{X\sim \mu^\ell}\left[ f_{\theta,Y,w}(X)\right]\right|\\
&=\frac{1}{L}\sum_{\ell=1}^L\underset{f\in \mathcal{F}_p^\text{s}}{\sup} \vert\mathbb{P}_n^\ell(f)-\mathbb{E}_{X\sim \mu^\ell}[f(X)]\vert,
\end{split}
\end{equation*}
where $\mathcal{F}_p^\text{s}$ is defined in Lemma \ref{lem:covering_number_F_p^s}. Moreover, we have for any $\ell\in[\![1,L]\!]$,
\begin{equation*}
\begin{split}
\sup_{f\in\mathcal{F}_p^s}\mathbb{P}^\ell_nf^2&=\sup_{(\theta,Y,w)\in \mathbb{S}^{d-1}\times\Br^N\times [-K,K]^N}\frac{1}{n}\sum_{k=1}^n \left(\min_{i=1,...,N}\vert \langle X_k^\ell-y_i,\theta\rangle\vert^p-w_i\right)^2\\
&\leq \frac{1}{n}\sum_{k=1}^n\left(\sup_{x,y\in\Br, \theta\in\mathbb{S}^{d-1}}\|x-y\|^p\Vert\theta\Vert^p+\sup_{w\in[-K,K]}w \right)^2\\
&\leq ((2R)^p+2pR^p)^2\leq(2(2R)^p)^2.
\end{split}
\end{equation*}
With the same argument as in the proof of case $ \mathbb{D}=W_p^p$, this quantity can be upper bounded in expectation using \cite[Theorem 3.5.1]{gine2021mathematical}:
\begin{equation*}
\mathbb{E}\left[\sqrt{n}\underset{f\in \mathcal{F}_p^s}{\sup} \mid\mathbb{P}_n^\ell(f)-\mathbb{E}_{X\sim \mu^\ell}[f(X)]\mid \right]\leq 8\sqrt{2} \mathbb{E}\left[\int_{0}^{2(2R)^p}\sqrt{\log\left(2\mathcal{N}(\tau/2,\mathcal{F}_p^{\text{s}},\|\cdot\|_{\mathbb{L}^2(\mu_n^\ell)})\right)}d\tau\right].
\end{equation*}
From Lemma \ref{lem:covering_number_F_p^s} we thus have
\begin{equation*}
\begin{split}
&\int_{0}^{2(2R)^p}\sqrt{\log\left(2\mathcal{N}(\tau/2,\mathcal{F}_p^\text{s},\|\cdot\|_{\mathbb{L}^2(\mu_n^\ell)})\right)}d\tau\\
&\leq \int_{0}^{2(2R)^p}\sqrt{\log\left(2\left(1+\frac{24pR^p}{\tau}\right)^d \left(1+\frac{12pR^{p}}{\tau}\right)^{dN} \left(\frac{24pR^p}{\tau}\right)^N\right)}d\tau\\
&\leq\sqrt{N(d+1)+d}\int_{0}^{2(2R)^p}\sqrt{\log\left(2+\frac{48pR^p}{\tau}\right)}d\tau.\\
\end{split}
\end{equation*}
Finally, as this last term does not depend on $\ell$, we get
\begin{equation*}
\mathbb{E}\left[\frac{1}{L}\sum_{\ell=1}^L\underset{f\in \mathcal{F}_p^\text{s}}{\sup} \vert\mathbb{P}_n^\ell(f)-\mathbb{E}_{X\sim \mu^\ell}[f(X)]\vert\right]\leq C_{p,R}\sqrt{\frac{N(d+1)+d}{n}},
\end{equation*}
where $C_{p,R} = \int_{0}^{2(2R)^p}\sqrt{\log\left(2+\frac{48pR^p}{\tau}\right)}d\tau$.

\end{proof}

\subsection{Entropy regularized barycenter}\label{sec:app_entropy}

\begin{lem}\label{lem_covering_number_entropic}
Let
\begin{align*}
    \mathcal{F}_p^\epsilon = \left\{
        \left\{ g_{Y,\pi,w}^\epsilon : x\in \Br \mapsto -\epsilon\log\left(\sum_{i=1}^N \pi_i \exp\left(\frac{w_i - \|x - y_i\|^p}{\epsilon}\right)\right)\right\}
    \right. \\
    \left. \middle| 
        Y \in \Br^N,\; \pi \in \Delta_N,\; w \in \overline{\text{B}_1\left(0,4p(2R)^p\right)}^N,
    \right\}
\end{align*}
 then for all $\tau>0$ and $\mu\in \mathcal{M}_1(\Br)$, we have
 $\mathcal{N}(\tau,\mathcal{F}_p^\epsilon,\|\cdot\|_{\mathbb{L}^2(\mu)})\leq \left(1+\frac{4pR^p}{\tau}\right)^{dN}\left(\frac{16p(2R)^p}{\tau}\right)^{N}.$
Also, the subset $\mathcal{F}_{p,\mathbb{Q}}^\epsilon$ of functions of $\mathcal{F}_p^\epsilon$ with rational parameters $(Y,\pi,w)$, is dense in $\mathcal{F}_p^\epsilon$.
\end{lem}
\begin{proof}
Let $(Y,\pi,w),\:(Z,\gamma,v)\in \Br^N\times \Delta_N\times [-4p(2R)^p,4p(2R)^p]^N $. Then for all $x\in \Br$, we have 
\begin{equation*}
\begin{split}
\vert g_{Y,\pi,w}^\epsilon(x)-g_{Z,\gamma,v}^\epsilon(x)\vert
&=\epsilon\left| \log\left(\frac{\sum_{i=1}^N\pi_i\exp\left(\frac{w_i-\|x-y_i\|^p}{\epsilon}\right)}{\sum_{i=1}^N\gamma_i\exp\left(\frac{v_i-\|x-z_i\|^p}{\epsilon}\right)}\right) \right|\\
&\leq \epsilon \left| \log\left(\frac{\max_{i=1,\ldots,N}\exp\left(\frac{w_i-\|x-y_i\|^p}{\epsilon}\right)\sum_{i=1}^N\pi_i}{\min_{i=1,\ldots,N}\exp\left(\frac{v_i-\|x-z_i\|^p}{\epsilon}\right)\sum_{i=1}^N\gamma_i}\right) \right|\\
&=\left|\max_{i=1,\ldots, N}w_i-\|x-y_i\|^p - \min_{i=1,\ldots,N}v_i-\|x-z_i\|^p\right|.
\end{split}
\end{equation*}
Denoting $\overline{i}$ (resp. $\underline{i}$) is an index for which the maximum (resp. minimum) 
of $i\mapsto w_i-\|x-y_i\|^p$ (resp. $i\mapsto v_i-\|x-z_i\|^p$) is attained, we get by Lemma \ref{lem:TAF-cout-norme},
\begin{equation}\label{ineq_mvt_entropic}
\vert g_{Y,\pi,w}(x)-g_{Z,\gamma,v}(x)\vert\leq \vert w_{\overline{i}}-v_{\underline{i}} \vert+ pR^{p-1}\|y_{\overline{i}}-z_{\underline{i}}\|.
\end{equation}
Notice that the right-hand side of the above inequality depends on neither $\pi$ nor $\epsilon$. Let $T=pR^{p-1}$, with the same proof argument as of case $\mathbb{D}=W_p^p$ and by boundedness of the entropic dual variables (Proposition \ref{Prop:barycentersproperties}) we thus get 
\begin{equation*}
\begin{split}
\mathcal{N}(\tau,\mathcal{F}_p^\epsilon,\|\cdot\|_{\mathbb{L}^2(\mu)})
&\leq \mathcal{N}\left(\frac{\tau}{2T},\overline{\text{B}_d(0,R)},\|\cdot\|\right)^N\mathcal{N}\left(\frac{\tau}{2},\overline{\text{B}_1(0,4p(2R)^p)},\mid\cdot\mid\right)^N\\
&\leq \left(1+\frac{4pR^p}{\tau}\right)^{dN}\left(\frac{16p(2R)^p}{\tau}\right)^{N}\\
\end{split}
\end{equation*}
which concludes the proof of the upper bound of the covering number. Finally, $\mathcal{F}_{p,\mathbb{Q}}^{\epsilon}$ is dense in  $\mathcal{F}_{p}^{\epsilon}$ due to inequality \eqref{ineq_mvt_entropic}.
\end{proof}

\begin{proof}[Proof of Theorem \ref{th:main} for $\mathbb{D}=W_{\epsilon,p}^p$]

We fix $\epsilon>0$. Let $(Y,\pi)\in \Br^N\times \Delta_N$ and let $M\in[\![1,N]\!]$ and $(\tilde{Y},\tilde{\pi})\in(\Br^M\setminus D_M)\times \mathrm{int}(\Delta_M)$ be as in Proposition \ref{prop:nondegenerated_version}. For any triplet $(Y,\pi,w)\in \Br^N\times\Delta_N\times\R^N$ and $x\in\Br$, we let
$g_{Y,\pi,w}^\epsilon(x):=-\epsilon\log\left(\sum_{j=1}^N\pi_j \exp\left( \frac{w_j-\|x-y_j\|^p}{\epsilon}\right)\right).$
For all $\ell\in [\![1,L]\!]$, let $X_1^\ell,\ldots,X_n^\ell$  be i.i.d random variables of law $\mu^\ell$. Using Proposition \ref{Prop:barycentersproperties} and the same argument as in the case $\mathbb{D}=W_p^p$, we get for $K:=4p(2R)^p$:
\begin{equation*}
\begin{split}
&\underset{(Y,\pi)\in \Br^N\times \Delta_N}{\sup}\left\vert \bary_\mathbb{D}\left(\mu_n^1,\ldots,\mu_n^L,\sum_{i=1}^N\pi_i\delta_{y_i}\right) - \bary_\mathbb{D}\left(\mu^1,\ldots,\mu^L,\sum_{i=1}^N\pi_i\delta_{y_i}\right) \right|\\
&\leq \frac{1}{L}\sum_{\ell=1}^L\ \underset{(Y,\pi,w)\in \Br^N\times  \Delta_N\times [-K,K]^N}{\sup}\vert  \mathbb{P}_n^\ell g_{Y,\pi,w}^\epsilon-\mathbb{E}_{X\sim \mu^\ell}\left[ g_{Y,\pi,w}^\epsilon(X)\right] \vert\\
&=\frac{1}{L}\sum_{\ell=1}^L\sup_{g\in \mathcal{F}_p^\epsilon}\vert \mathbb{P}_n^\ell(g)-\mathbb{E}_{X\sim \mu^\ell}[g(X)] \vert
\end{split}
\end{equation*}
where $\mathcal{F}_p^\epsilon$ is defined in Lemma \ref{lem_covering_number_entropic} . 
Moreover, for all $w\in\overline{\text{B}_1(0,4p(2R)^p)}^N$, we have 
\begin{align}
&-\frac{4p(2R)^p+(2R)^p}{\epsilon}\leq \frac{w_j-\|x-y_j\|^p}{\epsilon}\leq \frac{4p(2R)^p}{\epsilon}\nonumber\\
&\iff -\frac{(4p+1)(2R)^p}{\epsilon}\leq \log\left(\sum_{j=1}^N\pi_j\exp\left( \frac{w_j-\|x-y_j\|^p}{\epsilon}\right)\right) \leq  \frac{4p(2R)^p}{\epsilon}\nonumber\\
&\iff -4p(2R)^p\leq g_{Y,\pi,w}^\epsilon(x) \leq (4p+1)(2R)^p.\label{bound_on_g_eps}
\end{align}
Then, by Lemma \ref{lem_covering_number_entropic}, the bound obtained in \eqref{bound_on_g_eps} and, once again, with the same argument as in the case $\mathbb{D}=W_p^p$, we obtain by \cite[Theorem 3.5.1]{gine2021mathematical} the following bound
\begin{equation*}\label{ineq_covering_main_entropic}
\begin{split}
\mathbb{E}\sqrt{n}\underset{g\in \mathcal{F}_p^\epsilon}{\sup} \vert\mathbb{P}_n^\ell(g)&-\mathbb{E}_{X\sim \mu^\ell}[g(X)]\vert \leq 8\sqrt{2}\mathbb{E}\left[\int_{0}^{(4p+1)(2R)^p}\sqrt{\log\left(2\mathcal{N}(\tau/2,\mathcal{F}_p^\epsilon,\|\cdot\|_{\mathbb{L}^2(\mu_n^\ell)})\right)}d\tau\right]\\
&\leq 8\sqrt{2}\int_{0}^{(4p+1)(2R)^p}\sqrt{\log\left(2\left(1+\frac{8pR^p}{\tau}\right)^{dN}\left(\frac{32p(2R)^p}{\tau}\right)^{N}\right)}d\tau\\
&\leq 8\sqrt{2}\sqrt{N(d+1)}\int_{0}^{(4p+1)(2R)^p}\sqrt{\log\left(2+\frac{64p(2R)^p}{\tau} \right)}d\tau\\
\end{split}
\end{equation*}
so that finally
\begin{equation*}
\mathbb{E}\left[\frac{1}{L}\sum_{\ell=1}^L \ \underset{g\in \mathcal{F}_p^\epsilon}{\sup} \vert\mathbb{P}_n^\ell(g)-\mathbb{E}_{X\sim \mu^\ell}[g(X)]\vert\right]\leq C_{p,R}\sqrt{\frac{N(d+1)}{n}},
\end{equation*}
where $C_{p,R}=8\sqrt{2}\int_{0}^{(4p+1)(2R)^p}\sqrt{\log\left(2+\frac{64p(2R)^p}{\tau} \right)}d\tau$.
\end{proof}

In the following, we prove the result for debiased Sinkhorn barycenters.

\begin{proof}[Proof of Corollary \ref{cor:sinkhorn_debiased}]
Let $(Y^*,\pi^*)$ and $(Y^n,\pi^n)$ be respective minimizers of $(Y,\pi)\mapsto \Bary_\mathbb{D}\left(\mu^1,\ldots,\mu^L,\sum_{i=1}^N\pi_i\delta_{y_i}\right)$ and $(Y,\pi)\mapsto \Bary_\mathbb{D}\left(\mu_n^1,\ldots,\mu_n^L,\sum_{i=1}^N\pi_i\delta_{y_i}\right)$, which exist by continuity on the compact set $A$. We denote $\nu^*=\sum_{i=1}^N\pi_i^*\delta_{y_i^*}$ and $\nu^n=\sum_{i=1}^N\pi_i^n\delta_{y_i^n}$. Then 
\begin{equation*}
\begin{split}
&\bary_{\mathbb{D}}(\mu^1,\ldots,\mu^L,\nu^n)-\bary_{\mathbb{D}}(\mu^1,\ldots,\mu^L,\nu^*)\\
&\leq \bary_{\mathbb{D}}(\mu^1,\ldots,\mu^L,\nu^n)-\bary_{\mathbb{D}}(\mu_n^1,\ldots,\mu_n^L,\nu^n) + \bary_{\mathbb{D}}(\mu_n^1,\ldots,\mu_n^L,\nu^*)  -\bary_{\mathbb{D}}(\mu^1,\ldots,\mu^L,\nu^*)\\
&\leq2\underset{(Y,\pi)\in \Br ^N\times\Delta_N}{\sup}\left| \bary_{\mathbb{D}}\left(\mu_n^1,\ldots,\mu_n^L,\sum_{i=1}^N\pi_i\delta_{Y_i}\right) - \bary_{\mathbb{D}}\left(\mu^1,\ldots,\mu^L,\sum_{i=1}^N\pi_i\delta_{y_i}\right) \right|\\
&\leq \underset{(Y,\pi)\in\Br\times \Delta_N}{\sup} 2\left| \bary_{W_{\epsilon,p}^p}\left(\mu_n^1,\ldots,\mu_n^L, \sum_{i=1}^N\pi_i\delta_{y_i}\right)-\bary_{W_{\epsilon,p}^p}\left(\mu^1,\ldots,\mu^L, \sum_{i=1}^N\pi_i\delta_{y_i}\right) \right|\\
&+ \frac{1}{L}\sum_{\ell=1}^L\left| W_{\epsilon,p}^p(\mu_n^\ell,\mu_n^\ell)-W_{\epsilon,p}^p(\mu^\ell,\mu^\ell)  \right|.
\end{split}
\end{equation*}

In the right-hand side of the inequality, we bound the first term by Theorem \ref{th:main} while the second term can be bounded thanks to the proof of \cite[Theorem 3]{genevay2019sample}. In fact, their proof can be adapted almost directly to show that $\mathbb{E}\vert W_{\epsilon,p}^p(\mu_n,\mu_n)-W_{\epsilon,p}^p(\mu,\mu)\vert=O\left(\frac{e^\frac{{\mathcal{K}}}{\epsilon}}{\sqrt{n}}\left(1+\frac{1}{\epsilon^{\lfloor d/2\rfloor}} \right)\right).$ This result is given for arbitrary measures $\alpha,\beta\in\mathcal{M}_1(\R^d)$, their empirical counterparts $\hat{\alpha}_n,\:\hat{\beta}_n$ and the quantity $W_{\epsilon,p}^p(\hat{\alpha}_n,\hat{\beta}_n)$, but at no point in the proof the independence of the samples upon which $\hat{\alpha}_n$ and $\hat{\beta}_n$ are constructed is mandatory. As a result, one can obtain the same bound with $\beta=\alpha=\mu$.

\end{proof}

\section{Additional Lemmas}

This section includes elementary, yet useful Lemmas for the proof of our results. Additionally, for completeness, we prove a well known fact about optimal quantization that is used in several instances of our article. 

\subsection{Technical lemmas}
\begin{lem}\label{lem:TAF-cout-norme}
For $R>0$ and $p\geq1$, we have

\begin{itemize}
    \item[(i)] the function $f:x\in\Br\mapsto\|x\|^p$ is $pR^{p-1}$ Lipschitz,
    \item[(ii)] the function $c:(x,y)\in\Br\times \Br\mapsto \|x-y\|^p$ is $\sqrt{2}p(2R)^{p-1}$ Lipschitz.
\end{itemize} 
\end{lem}
\begin{proof} 

(i) The case $p=1$ is just the result of triangular inequality. Second, by mean value theorem and the chain rule, one can show that for $p>1$, 
\begin{equation*}
\nabla f(x)=p\|x\|^{p-2}x
\end{equation*}
whose norm is upper bounded by $pR^{p-1}$. 

(ii) Let $x,y,\in \Br$. We first consider the case $p=1$ :
\begin{equation*}
\begin{split}
\mid \|x-y\|-\|x'-y'\|\mid^2&=\mid \|x-y\|-\|x-y'\|-(\|x'-y'\|-\|x-y'\|)\mid^2\\
&\leq \mid\|x-y\|-\|x-y'\|\mid^2+ \mid \|x-y'\|-\|x'-y'\|\mid^2\\
&\leq\|x-x'\|^2+\|y-y'\|^2
\end{split}
\end{equation*}
by triangular inequality.

Fore the case $p>1$, we differentiate the functional $c:(x,y)\mapsto\|x-y\|^p$, that is 
\begin{equation*}
\begin{split}
\partial_xc(x,y)&=\frac{p}{2}(\|x-y\|^2)^{\frac{p}{2}-1}\partial_x\|x-y\|^2\\
&=p\|x-y\|^{p-2}(x-y).
\end{split}
\end{equation*}
It results that 
\begin{equation*}
\|\partial_xc(x,y)\|=p\|x-y\|^{p-1}
\end{equation*}
and therefore
\begin{equation*}
\begin{split}
\|\nabla c(x,y)\|^2&=\|\partial_xc(x,y)\|^2+\|\partial_yc(x,y)\|^2\\
&=2p^2\|x-y\|^{2p-2}.
\end{split}
\end{equation*}
We then conclude by mean value theorem and the fact that $x,y,\in \Br$.

\end{proof}

\begin{lem}\label{lem:ineq_sup}
For any functions $f,g:\R^d\longrightarrow\R$ and subset $U$ of $\R^d$
\begin{equation*}
\left|\sup_{x\in U}f(x)-\sup_{x\in U}g(x)\right|\leq\sup_{x\in U}\vert f(x)-g(x)\vert
\end{equation*}
and
\begin{equation*}
\left| \inf_{x\in U}f(x)-\inf_{x\in U}g(x)\right|\leq\sup_{x\in U}\vert f(x)-g(x)\vert
\end{equation*}
\end{lem}
\begin{proof}
Let $x\in U$, then 
\begin{equation*}
f(x)-\sup_{t\in U}g(t)\leq f(x)-g(x)\leq \sup_{x\in U}\vert f(x)-g(x)\vert.
\end{equation*}
This being true for all $x\in U$ means that 
\begin{equation*}
\sup_{x\in U}f(x)-\sup_{t\in U}g(x)\leq \sup_{x\in U}\vert f(x)-g(x) \vert.
\end{equation*}
By inverting the role of $f$ and $g$ we can then show the first result. As for the second result for all $\overline{x}\in U$ we have
\begin{equation*}
\begin{split}
\sup_{x\in U}\vert f(x) - g(x) \vert&\geq f(\overline{x})- g(\overline{x})\\
&\geq \inf_{x\in U} f(x)- g(\overline{x})
\end{split}
\end{equation*}
This being true for all $\overline{x}\in U$ means that
\begin{equation*}
\sup_{x\in U}\vert f(x) - g(x) \vert\geq \inf_{x\in U}f(x)- \inf_{x\in U}g(x).
\end{equation*}
By inverting the role of $f$ and $g$ we can show the second result.

\end{proof}

\subsection{Lower bound Lemma}\label{appendix_sec_lem_lowerbound}
In this subsection we prove that the worst-case lower bound given in \cite{bartlett2002minimax} also holds for Wasserstein sparse barycenters.
\begin{lem}\label{Lem:lowerbound_app}
Let $\mathbb{D}=W_2^2$, the number of measures $L\in \mathbb{N}^*$ and the dimension $d\in \mathbb{N}^*$. Assume  $N\geq 3$ and $n\geq 3435N$ and let  $\sum_{i=1}^N\pi_i^n\delta_{y_i^n}\in\mathcal M_1^N(\R^d)$ be a random measure depending on $nL$ random variables in $\Br$. Then
there exists $\mu^1,\ldots,\mu^L\in\mathcal{M}_1(\R^d)$ such that the estimation error of Corollary \ref{cor:sparse_ot_bar} verifies
\begin{equation}\label{eq:approximation_error_app}
\mathbb{E}\left[\bary_{\mathbb{D}}\left( \mu^1,\ldots,\mu^L,\sum_{i=1}^N\pi_i^n\delta_{y_i^n} \right)-\min_{(Y,\pi)\in \Br^N\times\Delta_N}\bary_{\mathbb{D}}\left( \mu^1,\ldots,\mu^L,\sum_{i=1}^N\pi_i\delta_{y_i} \right)\right]\geq C_d\sqrt{\frac{N^{1-\frac{4}{d}}}{n}}
\end{equation}
 where $C_{d}$ is an explicit constant which only depends on $d$.
\end{lem}

\begin{proof}
Let us first consider the case $L=1$.
By \cite[Theorem~1]{bartlett2002minimax}, there exists a measure
$\mu\in\mathcal{M}_1(\R^d)$ and an optimal quantizer
$(Y^*,\pi^*)\in\underset{(Y,\pi)\in(\R^d)^N\times\Delta_N}{\argmin}\ W_2^2\left(\mu,\sum_{i=1}^N\pi_i\delta_{y_i}\right)$ such that
\begin{equation*}
\mathbb{E}\!\left[
\min_{\pi\in \Delta_N}W_2^2\!\left(\mu,\sum_{i=1}^N\pi_i\delta_{y_i^n}\right)
- W_2^2\!\left(\mu,\sum_{i=1}^N\pi_i^*\delta_{y_i^*}\right)
\right]
\gtrsim \sqrt{\frac{N^{1-\frac{4}{d}}}{n}}.
\end{equation*}
Since almost surely
\[
\min_{\pi\in\Delta_N}W_2^2\!\left(\mu,\sum_{i=1}^N\pi_i\delta_{y_i^n}\right)
\le W_2^2\!\left(\mu,\sum_{i=1}^N\pi_i^n\delta_{y_i^n}\right),
\]
we deduce that
\begin{equation*}
\mathbb{E}\!\left[
W_2^2\!\left(\mu,\sum_{i=1}^N\pi_i^n\delta_{y_i^n}\right)
- W_2^2\!\left(\mu,\sum_{i=1}^N\pi_i^*\delta_{y_i^*}\right)
\right]
\gtrsim \sqrt{\frac{N^{1-\frac{4}{d}}}{n}}.
\end{equation*}
The result for $L>1$ in~\eqref{eq:approximation_error_app} then follows
by choosing $\mu^1=\cdots=\mu^L=\mu$.
\end{proof}

\subsection{Optimal quantization}

For completeness, we recall and prove this well known fact about optimal quantization.
\begin{lem}\label{K-means=Opt_quant}
Let $\mu$ be an arbitrary measure and $Y\in (\R^d)^N\setminus D_N,$ then for all $p\geq 1$
\begin{equation*}
\underset{\pi \in \Delta_N}{\min}W_p^p\left(\mu,\sum_{i=1}^N\pi_i\delta_{y_i}\right)=\int_{\mathbb{R}^d}\min_{i=1,\ldots,N}\|x-y_i\|^pd\mu(x).
\end{equation*}
\end{lem}

\begin{proof}
Consider the functional 
\[(w,\pi)\in \R^N\times \Delta_N \longmapsto \int_{\R^d}\min_{i=1,\ldots,N}\{\|x-y_i\|^p-w_i\}d\mu(x)+\sum_{i=1}^N \pi_iw_i.\]
It is concave in $w$ and convex in $\pi$, and we can therefore swap the min and the max in the following problem

\begin{equation}\label{last_lemma_eq}
    \min_{\pi\in \Delta_N}\max_{w\in \mathbb{R}^N} \int_{\R^d}\min_{i=1,\ldots,N}\{\|x-y_i\|^p-w_i\}d\mu(x)+\sum_{i=1}^N \pi_iw_i.
\end{equation}
Notice that a first order condition with respect to $\pi$ implies $w^*=0_{\R^N}$ at optimality. Let $\pi^*\in \Delta_N$ be a minimizer of \eqref{last_lemma_eq}. In order to use the duality expression \eqref{eq:kantorovich_semi_dual}, we consider  $\tilde{\pi}\in\text{int}(\Delta_M)$ the vector of weights $\pi^*$ whose null components have been removed. Following this construction and with a slight abuse of notation, we reorder the point cloud $Y:=(y_1,\ldots,y_N)$ so that the components $y_i$ corresponding to null weights have indexes $i\in [\![M+1,N]\!]$.

\begin{equation*}
\begin{split}
\underset{\pi \in \Delta_N}{\min}W_p^p\left(\mu,\sum_{i=1}^N\pi_i\delta_{y_i}\right)&=W_p^p\left(\mu,\sum_{i=1}^M\tilde{\pi}_i\delta_{y_i}\right)\\
&=\min_{\pi\in \Delta_M}\max_{w\in\R^M}\int_{\R^d}\min_{i=1,\ldots,M}\{\|x-y_i\|^p-w_i\} d\mu(x)+\sum_{i=1}^M\pi_iw_i\\
&=\max_{w\in\R^M}\min_{\pi\in \Delta_M}\int_{\R^d}\min_{i=1,\ldots,M}\{\|x-y_i\|^p-w_i\} d\mu(x)+\sum_{i=1}^M\pi_iw_i\\
&=\int_{\R^d}\min_{i=1,\ldots,M}\|x-y_i\|^pd\mu(x).
\end{split}
\end{equation*}
We can then consider two cases. First, suppose that $\mu$ does not give mass to the Voronoï set $\{x\in\R^d\:|\:\|x-y_i\|^p\leq \|x-y_j\|^p,\:j\in[\![1,N]\!]\}$ for all $i\in[\![M+1,N]\!]$; this may be the case for example if $\mu$ is discrete. Then we have
\begin{equation*}
\int_{\R^d}\min_{i=1,\ldots,M}\|x-y_i\|^pd\mu(x)=\int_{\R^d}\min_{i=1,\ldots,N}\|x-y_i\|^pd\mu(x).
\end{equation*}

Second, suppose that there exists $i\in[\![M+1,N]\!]$ such that $\mu$ gives mass to a set $\{x\in\R^d\:|\:\|x-y_i\|^p\leq \|x-y_j\|^p,\:j\in[\![1,N]\!]\}$. By contradiction, suppose that $M<N$. Because $Y\notin D_N$ we then have
\begin{equation*}
\begin{split}
\int_{\R^d}\min_{i=1,\ldots,M}\|x-y_i\|^pd\mu(x)&>\int_{\R^d}\min_{i=1,\ldots,N}\|x-y_i\|^pd\mu(x)\\
&=\max_{w\in\R^N}\min_{\pi\in\Delta_N}\int_{\R^d}\min_{i=1,\ldots,N}\{\|x-y_i\|^p-w_i\} d\mu(x)+\sum_{i=1}^N\pi_iw_i\\
&=\underset{\pi \in \Delta_N}{\min}W_p^p\left(\mu,\sum_{i=1}^N\pi_i\delta_{y_i}\right)
\end{split}
\end{equation*}
which contradicts the optimality of the minimum in $\pi$. We thus get the result.
\end{proof}

\bibliographystyle{plain}
\bibliography{ref}

\end{document}